\documentclass[12pt]{article}
\usepackage{amssymb}
\usepackage{mathrsfs}
\usepackage[hypertex]{hyperref}
\usepackage{amsfonts}
\usepackage{graphicx}
\usepackage{mathptmx}
\usepackage{latexsym,amsmath,amssymb,amsfonts,amsthm}

\newtheorem{theorem}{Theorem}[section]
\newtheorem{lemma}[theorem]{Lemma}

\newtheorem{corollary}[theorem]{Corollary}
\newtheorem{proposition}[theorem]{Proposition}
\newtheorem{remark}[theorem]{Remark}
\newtheorem{example}[theorem]{Example}

\newtheorem{defn}[theorem]{Definition}

\begin{document}
\setcounter{page}{1}
\title{Eigenvalue inequalities for the $p$-Laplacian on a Riemannian manifold and estimates for the heat kernel}
\author{Jing Mao$^{1,2}$}
\date{}
\protect\footnotetext{\!\!\!\!\!\!\!\!\!\!\!\!{ MSC 2010: 35J60;
35P15; 58C40}
 }
\maketitle ~~~\\[-15mm]
\begin{center}{\footnotesize
$^{1}$Centro de F\'{\i}sica das Interac\c c\~{o}es Fundamentais,
Instituto Superior T\'{e}cnico, Technical University of Lisbon,
Edif\'{\i}cio Ci\^{e}ncia, Piso 3, Av. Rovisco Pais, 1049-001
Lisboa, Portugal; jiner120@163.com, jiner120@tom.com

$^{2}$Departamento de Matem\'{a}tica, Instituto Superior
T\'{e}cnico, Technical University of Lisbon, Edif\'{\i}cio
Ci\^{e}ncia, Piso 3, Av.\ Rovisco Pais, 1049-001 Lisboa, Portugal}
\end{center}

\begin{abstract}
  In this paper, we successfully generalize the eigenvalue comparison
theorem for the Dirichlet $p$-Laplacian ($1<p<\infty$) obtained by
Matei [A.-M. Matei, First eigenvalue for the $p$-Laplace operator,
\emph{Nonlinear Anal. TMA} {\bf 39} (8) (2000) 1051--1068] and
Takeuchi [H. Takeuchi, On the first eigenvalue of the $p$-Laplacian
in a Riemannian manifold, \emph{Tokyo J. Math.} {\bf 21} (1998)
135--140], respectively. Moreover, we use this generalized
eigenvalue comparison theorem to get estimates for the first
eigenvalue of the Dirichlet $p$-Laplacian of geodesic balls on
complete Riemannian manifolds with \emph{radial Ricci curvature
bounded from below w.r.t. some point}. In the rest of this paper, we
derive an upper and lower bound for the heat kernel of geodesic
balls of complete manifolds with specified curvature constraints,
which can supply new ways to prove the most part of two generalized
eigenvalue comparison results given by Freitas, Mao and Salavessa in
[P. Freitas, J. Mao and I. Salavessa, Spherical symmetrization and
the first eigenvalue of geodesic disks on manifolds, submitted
(2012)].
\end{abstract}

\markright{\sl\hfill  J. Mao \hfill}

\section{Introduction}
\renewcommand{\thesection}{\arabic{section}}
\renewcommand{\theequation}{\thesection.\arabic{equation}}
\setcounter{equation}{0} \setcounter{maintheorem}{0}

By using the theory of self-adjoint operators, the spectral
properties of the linear Laplacian on a domain in a Euclidean space
or a manifold have been studied extensively. Mathematicians
generally are interested in the spectrum of the Laplacian on compact
manifolds (with or without boundary) or noncompact complete
manifolds, since in these two cases the linear Laplacians can be
uniquely extended to self-adjoint operators (cf. \cite{mp2,mp1}).
However, the spectrum of the Laplacian on noncompact noncomplete
manifolds also attracts attention of mathematicians and physicists
in the past three decades, since the study of the spectral
properties of the Dirichlet Laplacian in infinitely stretched
regions has applications in elasticity, acoustics, electromagnetism,
quantum physics, etc. Recently, the author has proved the existence
of discrete spectrum of the linear Laplacian on a class of
$4$-dimensional rotationally symmetric quantum layers, which are
noncompact noncomplete manifolds, in \cite{m} under some geometric
assumptions therein.

A natural generalization of the linear Laplacian is the so-called
$p$-Laplacian below. Although many results about the linear
Laplacian ($p=2$) have been obtained, many rather basic questions
about the spectrum of the nonlinear $p$-Laplacian remain to be
solved.

Let $\Omega$ be a bounded domain on an $n$-dimensional Riemannian
manifold $(M,g)$. We consider the following nonlinear Dirichlet
eigenvalue problem
\begin{eqnarray*}
\left\{
\begin{array}{ll}
\Delta_{p}u+\lambda|u|^{p-2}u=0  ~\qquad {\rm{in}}~\Omega,\\
  u=0   \qquad \qquad \qquad \qquad {\rm{on}}~\partial\Omega,  & \quad
  \end{array}
\right.
\end{eqnarray*}
where $\Delta_{p}u={\rm{div}}(|\nabla{u}|_{g}^{p-2}\nabla{u})$ is
the $p$-Laplacian with $1<p<\infty$. In local coordinates
$\{x_{1},\ldots,x_{n}\}$ on $M$, we have
\begin{eqnarray} \label{1.1}
\Delta_{p}u=\frac{1}{\sqrt{\det(g_{ij})}}\sum\limits_{i,j=1}^{n}\frac{\partial}{\partial{x_{i}}}\left(\sqrt{\det(g_{ij})}g^{ij}|\nabla{u}|^{p-2}
\frac{\partial{u}}{\partial{x_{j}}}\right),
\end{eqnarray}
where
$|\nabla{u}|^2=|\nabla{u}|^2_{g}=\sum\limits_{i,j=1}^{n}g^{ij}\frac{\partial{u}}{\partial{x_{i}}}\frac{\partial{u}}{\partial{x_{j}}}$,
and $(g^{ij})=(g_{ij})^{-1}$ is the inverse of the metric matrix. A
well-known result about the above nonlinear eigenvalue problem
states that it has a positive weak solution, which is unique modulo
the scaling, in the space $W^{1,p}_{0}(\Omega)$, the completion of
the set $C^{\infty}_{0}(\Omega)$ of smooth functions compactly
supported on $\Omega$ under the Sobolev norm
$\|u\|_{1,p}=\{\int_{\Omega}(|u|^p+|\nabla{u}|^p)d\Omega\}^{\frac{1}{p}}$.
For a bounded simply connected domain with sufficiently smooth
boundary in Euclidean space, one can get a simple proof of this fact
in \cite{mb}. Moreover, the first Dirichlet eigenvalue
$\lambda_{1,p}(\Omega)$ of the $p$-Laplacian can be characterized by
\begin{eqnarray} \label{1.2}
\lambda_{1,p}(\Omega)=\inf\left\{\frac{\int_{\Omega}|\nabla{u}|^{p}d\Omega}{\int_{\Omega}|u|^{p}d\Omega}\Big{|}
u\neq0, u\in{W}^{1,p}_{0}(\Omega)\right\}.
\end{eqnarray}

By using spherically symmetric manifolds as the model spaces and
applying a similar method to that of the proof of theorem 3.6 in
\cite{fmi}, we give a Cheng-type eigenvalue comparison result for
the first eigenvalue of the $p$-Laplace operator in Section 3 -- see
Theorem \ref{theorem1} for the precise statement.

Besides the $p$-Laplacian, we also investigate the heat equation in
this paper. Given an $n$-dimensional Riemannian manifold $M$ with
associated Laplace-Beltrami operator $\Delta$. Then we are able to
define a differential operator $L$, which is known as the heat
operator, by
\begin{eqnarray*}
L=\Delta-\frac{\partial}{\partial{t}}
\end{eqnarray*}
acting on functions in $C^{0}\left(M\times(0,\infty)\right)$, which
are $C^{2}$ w.r.t. the variable $x$, varying on $M$, and $C^{1}$
w.r.t. the variable $t$, varying on $(0,\infty)$. Correspondingly,
the heat equation is given by
\begin{eqnarray} \label{heatequation}
Lu=0  \quad \left({\rm{equivalently,}~}
\Delta{u}-\frac{\partial{u}}{\partial{t}}=0\right),
\end{eqnarray}
with $u\in{C^{0}\left(M\times(0,\infty)\right)}$. The heat equation,
which can be used to describe the conduction of heat through a given
medium, and related deformations of the heat equation, like the
diffusion equation, the Fokker-Planck equation, and so on, are of
basic importance in variable scientific fields.

In fact, by applying volume comparison results proved by Freitas,
Mao and Salavessa in \cite{fmi}, we can obtain an upper and lower
bound for the heat kernel, which can be seen as an extension to the
existing results -- see Theorem \ref{71theoremheatchapter2} for the
precise statement.

The paper is organized as follows. In the next section, we will give
some preliminary knowledge on the model spaces. Theorem
\ref{theorem1} will be proved in Section 3. By using Theorem
\ref{theorem1}, some estimates for the first eigenvalue of the
Dirichlet $p$-Laplacian of a geodesic ball on a complete Riemannian
manifold with \emph{a radial Ricci curvature lower bound w.r.t. some
point} will be given in Section 4. Some fundamental truths about the
heat equation will be listed in Section 5. In Section 6, we will
prove Theorem \ref{71theoremheatchapter2} and give new ways to prove
the most part of two generalized eigenvalue comparison results in
\cite{fmi}. In fact, this paper is based on a part (Section 2.7 of
Chapter 2, Chapter 3) of the author's Ph.D. thesis
\cite{PhDJingMao}.

\section{Geometry of the model spaces and generalized Bishop's volume comparison
results}  \label{section2}
\renewcommand{\thesection}{\arabic{section}}
\renewcommand{\theequation}{\thesection.\arabic{equation}}
\setcounter{equation}{0} \setcounter{maintheorem}{0}
 One of the purposes of this paper is to give some inequalities for the first
eigenvalue of the $p$-Laplace operator. In order to state our
results here, we need to use some notions below, which have been
introduced in \cite{fmi,PhDJingMao} in detail.

For any point $q$ on an $n$-dimensional ($n\geq2$) complete
Riemannian manifold $M$ with the metric
$\langle\cdot,\cdot\rangle_M$ and the Levi-Civita connection
$\nabla$, we can set up a geodesic polar coordinates $(t,\xi)$
around this point $q$, where $\xi\in{S}_{q}^{n-1}\subseteq{T_{q}M}$
is a unit vector of the unit sphere $S_{q}^{n-1}$ with center $q$ in
the tangent space $T_{q}M$. Let $\mathcal{D}_{q}$, a star shaped set
of $T_{q}M$, and $d_{\xi}$ be defined by
\begin{eqnarray*}
\mathcal{D}_{q}=\{t\xi|~0\leq{t}<d_{\xi},~\xi\in{S^{n-1}_{q}}\},
\end{eqnarray*}
and
\begin{eqnarray*}
d_{\xi}=d_{\xi}(q):=  \sup\{t>0|~\gamma_{\xi}(s):= \exp_q(s\xi)~
{\rm{is~ the~ unique~ minimal~ geodesic ~joining} }~ q ~{\rm{and}}
~\gamma_{\xi}(t)\}.
\end{eqnarray*}
Then $\exp_q:\mathcal{D}_q \to M\backslash Cut(q)$ is a
diffeomorphism from $\mathcal{D}_q$ onto the open set $M\backslash
Cut(q)$, with $Cut(q)$ the cut locus of  $q$, which is a closed set
of zero $n$-Hausdorff measure. For $\eta\in{\xi^{\bot}}$, we can
define so-called the path of linear transformations
$\mathbb{A}(t,\xi):\xi^\perp\rightarrow{\xi^\perp}$ by
 \begin{eqnarray*}
\mathbb{A}(t,\xi)\eta=(\tau_{t})^{-1}Y(t),
 \end{eqnarray*}
with $\xi^\perp$ the orthogonal complement of $\{\mathbb{R}\xi\}$ in
$T_{q}M$, where $\tau_{t}:T_{q}M\rightarrow{T_{\exp_{q}(t\xi)}M}$ is
the parallel translation along the geodesic $\gamma_{\xi}(t)$ with
$\gamma'(0)=\xi$, and $Y(t)$ is the Jacobi field along
$\gamma_{\xi}$ satisfying $Y(0)=0$, $(\triangledown_{t}Y)(0)=\eta$.
Moreover, set
\begin{eqnarray*}
\mathcal{R}(t)\eta=(\tau_{t})^{-1}R(\gamma'_{\xi}(t),\tau_{t}\eta)
\gamma'_{\xi}(t),
\end{eqnarray*}
where the curvature tensor $R(X,Y)Z$ is defined by
$R(X,Y)Z=-[\nabla_{X},$ $ \nabla_{Y}]Z+ \nabla_{[X,Y]}Z$. Then
$\mathcal{R}(t)$ is a self-adjoint operator on $\xi^{\bot}$, whose
trace is the radial Ricci tensor
$$\mathrm{Ricci}_{\gamma_{\xi}(t)}(\gamma'_{\xi}(t),\gamma'_{\xi}(t)).$$
Clearly, the map $\mathbb{A}(t,\xi)$ satisfies the Jacobi equation
 $\mathbb{A}''+\mathcal{R}\mathbb{A}=0$ with initial conditions
 $\mathbb{A}(0,\xi)=0$, $\mathbb{A}'(0,\xi)=I$, and  by applying Gauss's lemma the Riemannian metric of $M$ can be
 expressed by
 \begin{eqnarray} \label{2.1}
 ds^{2}(\exp_{q}(t\xi))=dt^{2}+|\mathbb{A}(t,\xi)d\xi|^{2}
 \end{eqnarray}
on the set $\exp_{q}(\mathcal{D}_q)$. We consider the metric
components  $g_{ij}(t,\xi)$, $i,j\geq 1$, in a coordinate system
$\{t, \xi_a\}$ formed by fixing  an orthonormal basis $\{\eta_a,
a\geq 2\}$ of
 $\xi^{\bot}=T_{\xi}S^{n-1}_q$, and extending it to a local frame $\{\xi_a, a\geq2\}$ of
$S_q^{n-1}$. Define
 a function $J>0$ on $\mathcal{D}_{q}$ by
\begin{equation} \label{J}
J^{n-1}=\sqrt{|g|}:=\sqrt{\det[g_{ij}]}.
\end{equation}
 Since
 $\tau_t: S_q^{n-1}\to S_{\gamma_{\xi}(t)}^{n-1}$ is an
isometry, we have
$$
\langle d(\exp_q)_{t\xi}(t\eta_{a}),
d(\exp_q)_{t\xi}(t\eta_{b})\rangle_{M} =\langle
\mathbb{A}(t,\xi)(\eta_{a}), \mathbb{A}(t,\xi)(\eta_{b})\rangle_{M},
$$
and so,
$$ \sqrt{|g|}=\det\mathbb{A}(t,\xi).$$
So, by applying (\ref{2.1}) and (\ref{J}), the volume $V(B(q,r))$ of
a geodesic ball $B(q,r)$, with radius $r$ and center $q$, on $M$ is
given by
\begin{eqnarray} \label{2.2}
V(B(q,r))=\int_{S_{q}^{n-1}}\int_{0}^{\min\{r,d_{\xi}\}}\sqrt{|g|}dtd\sigma=\int_{S_{q}^{n-1}}\left(\int_{0}^{\min\{r,d_{\xi}\}}\det(\mathbb{A}(t,\xi))dt\right)d\sigma,
\end{eqnarray}
where $d\sigma$ denotes the $(n-1)$-dimensional volume element on
$\mathbb{S}^{n-1}\equiv S_{q}^{n-1}\subseteq{T_{q}M}$. Let
$inj(q):=d(q,Cut(q))=\min_{\xi}d_{\xi}$ be the injectivity radius at
$q$. In general, we have
$B(q,inj(q))\subseteq{M}\backslash{Cut(q)}$. Besides, for
$r<inj(q)$, by (\ref{2.2}) we can obtain
\begin{eqnarray*}
V(B(q,r))=\int_{0}^{r}\int_{S_{q}^{n-1}}\det(\mathbb{A}(t,\xi))d\sigma{dt}.
\end{eqnarray*}
Denote by $r(x)=d(x,q)$ the intrinsic distance to the point
$q\in{M}$. Then, by the definition of a non-zero tangent vector
``\emph{radial}" to a prescribed point on a manifold given in the
first page of \cite{KK}, we know that for
$x\in{M}\backslash(Cut(q)\cup{q})$ the unit vector field
\begin{eqnarray*}
v_{x}:=\nabla{r(x)}
\end{eqnarray*}
is the radial unit tangent vector at $x$. This is because for any
$\xi\in{S}_{q}^{n-1}$ and $t_{0}>0$, we have
$\nabla{r}{(\gamma_{\xi}(t_{0}))}=\gamma'_{\xi}(t_{0})$ when the
point $\gamma_{\xi}(t_{0})=\exp_{q}(t_{0}\xi)$ is away from the cut
locus of $q$ (cf. \cite{a2}). Set
\begin{eqnarray} \label{important}
l(q):=\sup_{x\in{M}}r(x),
\end{eqnarray}
 Then we have $l(q)=\max_{\xi}d_{\xi}$
(cf. Section 2 of \cite{fmi}). Clearly, $l(q)\geq{inj(q)}$. We also
need the following  fact about $r(x)$ (cf. \cite{p}, Prop. 39 on p.\
266),
\begin{eqnarray*}
\partial_{r}\Delta{r}+\frac{(\Delta{r})^2}{n-1}\leq\partial_{r}\Delta{r}+|{\rm{Hess}}r|^{2}=-{\rm{Ricci}}(\partial_{r},\partial_{r}),
\qquad {\rm{with}}~~\Delta{r}=\partial_{r}\ln( \sqrt{|g|}),
\end{eqnarray*}
with $\partial_{r}=\nabla{r}$ as a differentiable vector (cf.
\cite{p}, Prop. 7 on p.\ 47 for the differentiation of
$\partial_{r}$).
 Then, together with (\ref{J}), we have
\begin{eqnarray}
&& J''+\frac{1}{(n-1)}{\rm{Ricci}}\left(\gamma'_{\xi}(t),
\gamma'_{\xi}(t)\right)J\leq 0,  \label{J''}\\
&&J(0,\xi)=0,\quad J'(0,\xi)=1 \label{J0}.
\end{eqnarray}
The facts (\ref{J''}) and (\ref{J0}) make a fundamental role in the
derivation of the so-called \emph{generalized Bishop's volume
comparison theorem I} below (cf. \cite{fmi,PhDJingMao}).

We use spherically symmetric manifolds as our model spaces, which
can be defined as follows.
\begin{defn} (\cite{fmi,PhDJingMao}) \label{def1}
A domain
$\Omega=\exp_{q}([0,l)\times{S_{q}^{n-1}})\subset{M}\backslash
Cut(q)$, with $l<inj(q)$, is said to be spherically symmetric with
respect to a point $q\in\Omega$, if the matrix $\mathbb{A}(t,\xi)$
satisfies $\mathbb{A}(t,\xi)=f(t)I$, for a function
$f\in{C^{2}([0,l])}$, $l\in(0,\infty]$ with $f(0)=0$, $f'(0)=1$,
$f|(0,l)>0$.
\end{defn}

 So, by (\ref{2.1}), on the set
 $\Omega$ given in Definition \ref{def1} the Riemannian metric of $M$ can be
 expressed by
 \begin{eqnarray} \label{2.3}
 ds^{2}(\exp_{q}(t\xi))=dt^{2}+f(t)^{2}|d\xi|^{2}, \qquad
 \xi\in{S_{q}^{n-1}}, \quad 0\leq{t}<l,
 \end{eqnarray}
 with $|d\xi|^{2}$ the round metric on the unit sphere $\mathbb{S}^{n-1}\subseteq\mathbb{R}^{n}$.
 Spherically symmetric manifolds were named as generalized space
forms
 by Katz and Kondo \cite{KK}, and a standard model for
such manifolds is given by the quotient manifold of the warped
product $[0,l)\times_{f} \mathbb{S}^{n-1}$ equipped with the metric
(\ref{2.3}), where $f$ satisfies the conditions of Definition
\ref{def1}, and all pairs $(0,\xi)$ are identified with a single
point $q$ (see \cite{cg1}). More precisely, an $n$-dimensional
spherically symmetric manifold $M^{\ast}$ satisfying those
conditions in Definition \ref{def1} is a quotient space
$M^{\ast}=\left([0,l)\times_{f(t)}\mathbb{S}^{n-1}\right)/\sim$ with
the equivalent relation ``$\sim$" given by
\begin{eqnarray*}
(t,\xi)\sim(s,\eta)\Longleftrightarrow\left\{
\begin{array}{lll}
t=s\quad {\rm{and}} \quad \xi=\eta, \\
or\\
t=s=0.  & \quad
\end{array}
\right.
\end{eqnarray*}
This relation is natural, and we can just use
$[0,l)\times_{f(t)}\mathbb{S}^{n-1}$ to represent this quotient.
That is to say, $M^{\ast}=[0,l)\times_{f(t)}\mathbb{S}^{n-1}$ with
$f(t)$ satisfying conditions in Definition \ref{def1} is a
spherically symmetric manifold with $q$ the base point and
(\ref{2.3}) as its metric. This metric is of class $C^{k}$,
$k\geq0$, if $f\in{C}^{k}((0,l))$ and of class $C^{k+3}$ at $t=0$,
with vanishing $2d$-derivatives (i.e. even-order derivatives or
derivatives of order $2d$) at $t=0$ for all $2d\leq{k}+3$ (see
\cite{p} p.13). Besides, if $l=+\infty$, then $M^{\ast}$ has a pole
at $p=\{0\}\times_{f}\mathbb{S}^{n-1}$, and vice versa. If
$l=+\infty$ and the metric is of class $C^2$, then by proposition 38
of chapter 7 in \cite{k}, we know that geodesics emanating from $q$
are defined for all $t\in\mathbb{R}$, which implies that $M^{\ast}$
is complete by the Hopf-Rinow theorem. If $l$ is finite and
$f(l)=0$, then $M^{\ast}$ ``closes". Besides, we are able to define
a one-point compactification metric space
$\overline{M}^{\ast}=M^{\ast}\cup\{q^{\ast}\}$ by identifying all
pairs $(l,\xi)$ with a single point $q^{\ast}$, and extending the
distance function to $q^{\ast}$ such that $d(q^{\ast},(t,\xi))=l-t$,
where, for a fixed $t$, $(t,\xi)$ can be used to represent a
geodesic sphere $\partial{B}(q,t)$ of radius $t$ centered at $q$.
Furthermore, if the metric (\ref{2.3}) can be extended continuously
to the closing point, that is, at $t=l$, $f$ is $C^3$ with
$f'(l)=-1$ and $f''(l)=0$, then this one-point compactification
metric space will be a Riemannian metric space. As the case of
$t=0$, if $f$ is of class $C^{k+3}$ ($k\geq0$) at $t=l$, with
vanishing $2d$-derivatives at $t=l$ for all $2d\leq{k}+3$ (of
course, $f(l)=0$, $f'(l)=-1$ are included here), then the metric is
of $C^{k}$ at the closing point $t=l$. Arguments similar to this
part about the regularity of the model spaces, spherically symmetric
manifolds, can also be found in \cite{fmi,PhDJingMao}, but we still
would like to recall these fundamental geometric properties here,
which are necessary and convenient for us to explain and try to
prove the results of this paper. For $M^{\ast}$ and $r<l$, by
(\ref{2.2}) we have
\begin{eqnarray*}
V(B(q,r))=w_{n}\int_{0}^{r}f^{n-1}(t)dt,
 \end{eqnarray*}
  and moreover, by applying the co-area formula, the volume of the boundary $\partial{B(q,r)}$ is given by
\begin{eqnarray*}
V(\partial{B(q,r)})=w_{n}f^{n-1}(r),
\end{eqnarray*}
where $w_{n}$ denotes the $(n-1)$-volume of the unit sphere
$\mathbb{S}^{n-1}\subseteq{\mathbb{R}^{n}}$. A space form with
constant curvature $k$ is also a spherically symmetric manifold, and
in this special case we have
\begin{eqnarray*}
f(t)=\left\{
\begin{array}{llll}
\frac{\sin\sqrt{k}t}{\sqrt{k}}, & \quad  l= \frac{\pi}{\sqrt{k}}
  & \quad k>0,\\
 t, &\quad l=+\infty & \quad k=0, \\
\frac{\sinh\sqrt{-k}t}{\sqrt{-k}}, & \quad l=+\infty  &\quad k<0.
\end{array}
\right.
\end{eqnarray*}

Under some constraints on the regularity of the warping function
$f$, Freitas, Mao and Salavessa have proved an asymptotical property
for the first eigenvalue of the linear Laplacian on spherically
symmetric manifolds (cf. lemma 2.5 in \cite{fmi}). By using a
similar method, we can improve it to the case of the nonlinear
Laplace operator as follows.

\begin{lemma} \label{closedmodel}
 Assume $M$ is a generalized  space form
$[0,l)\times_f \mathbb{S}^{n-1}$ (with $q\in{M}$ as its base point)
with $f\in C^{2}([0,l))$ and  $C^3$ at $t=0$, $f(0)=f''(0)=0$,
$f'(0)=1$, closing at $t=l$, i.e. $f(l)=0$. We have

(I) in case $n=2$, if for some $\epsilon >0$, $ f\in C^1([0,
l+\epsilon))$, then $\lim_{r\to l^-} \lambda_{1,p}(B(q,r))=0$ with
$1<p\leq2$;

(II) in case $n\geq3$, if for some $\epsilon >0$, $f\in
C^{2}([0,l+\epsilon))$,
 then $\lim_{r\to l^-}
\lambda_{1,p}(B(q,r))=0$ with $1<p<3$.
\end{lemma}
\begin{proof}
Here we would like to follow the idea of lemma 2.5 in \cite{fmi} to
prove our lemma. More precisely, we try to find a sequence
$\{\phi_{m}\}$  with $\phi_{m}\in{W}^{1,p}_{0}\left(B(q,r)\right)$
such that $\|\phi_{m}-1\|_{1,p}\rightarrow0$ as $m\rightarrow\infty$
and $r\rightarrow{l^-}$, and $\nabla\phi_{m}$ converges to $0$ for
the same norm as $m\rightarrow\infty$ and $r\rightarrow{l^-}$. Then,
together with (\ref{1.2}), we have $\lim_{r\to l^-}
\lambda_{1,p}(B(q,r))=0$. Denote by $B_{r}:=B(q,r)$ for $r<l$, which
has a $C^2$ boundary, and by $B_{l}=M$. Set
$V(r):=|B_r|=\int_{B_r}1$. For any increasing sequence $\{R_{m}\}$
with $R_{m}\uparrow{l}$, $R_m<R_{m+1}<l$, as in \cite{fmi}, we can
define a continuous function $y_m:[0,l)\rightarrow[0,1]$, which is
given by
\begin{eqnarray*}
y_{m}(r)=\left\{
\begin{array}{lll}
1, & \quad 0\leq{r}<R_m, \\
\frac{\ln\left(\frac{l-r}{l-R_{m+1}}\right)}{\ln\left(\frac{l-R_m}{l-R_{m+1}}\right)},
& \quad R_{m}\leq{r}\leq{R_{m+1}}, \\
 0, & \quad R_{m+1}<r<l,
\end{array}
\right.
\end{eqnarray*}
for $n=2$, and
\begin{eqnarray*}
y_m(r)=\left\{\begin{array}{lll}
 1, & \quad 0 \leq r<  R_m, \\
\frac{R_{m+1}-r}{R_{m+1} -R_m}, & \quad R_m\leq r \leq
R_{m+1},\\
0, & \quad R_{m+1}< r<l,
\end{array}
\right.
\end{eqnarray*}
for $n\geq3$. Clearly,
$\phi_{m}(x):=y_{m}(r(x))\in{W}^{1,p}_{0}\left(B_{R_{m+1}}\right)$,
where $r(x)=d(q,x)$ is the distance to $q$ for $x\in{M}$. Recall
that $r(x)$ is Lipschitz continuous on all $M$ with
$|\nabla{r}|\leq1$ a.e..

Assume that $n=2$. By the assumptions on $f$ and the Taylor's
formula, we have $f(s)=\eta(s)(s-l)$ with
$\eta(s):=\int_{0}^{1}f'(l+t(s-l))dt$ a bounded function for $s$
close to $l$. Without loss of generality, choose
$\alpha_{m}=\frac{1}{m!}$ and let $R_m=l-\alpha_{m}$. Therefore, for
$1<p\leq2$, we have
\begin{eqnarray*}
\int_{M}|\phi_{m}-1|^{p}\leq\int_{M\backslash{B_{R_m}}}1^p=|M|-V(R_m)\rightarrow0,
\qquad {\rm{as}}~~ m\rightarrow\infty.
\end{eqnarray*}
Besides, since for $s$ close to $l$, $\eta(s)$ is bounded, there
exists a constant $B_1>0$ such that for $m$ large enough, we have
$|\eta(s)|\leq{B_1}$, which implies
\begin{eqnarray*}
\int_{M}|\nabla(\phi_{m}-1)|^{p}&\leq&\frac{2\pi{B_1}}{\left(\ln\left(\frac{l-R_m}{l-R_{m+1}}\right)\right)^{p}}
\int_{R_{m}}^{R_{m+1}}\frac{1}{|l-s|^{p}}(l-s)ds\\
&=& \frac{2\pi{B_1}}{\ln\left(\frac{l-R_m}{l-R_{m+1}}\right)}=
\frac{2\pi{B_1}}{\ln(m+1)}\rightarrow0, \quad {\rm{as}}
~~m\rightarrow\infty, \qquad ({\rm{when}}~p=2);\\
 && {\rm{or}} \\
&&\frac{2\pi{B_1}}{\left(\ln\left(\frac{l-R_m}{l-R_{m+1}}\right)\right)^{p}}\cdot\frac{\left(l-R_{m}\right)^{2-p}-\left(l-R_{m+1}\right)^{2-p}}{2-p}\\
&&=\frac{2\pi{B_1}\left[\left(\frac{1}{m!}\right)^{2-p}-\left(\frac{1}{(m+1)!}\right)^{2-p}\right]}{(2-p)\left(\ln(m+1)\right)^{p}}\rightarrow0,
\quad {\rm{as}}~~m\rightarrow\infty, \quad ({\rm{when}}~1<p<2).
\end{eqnarray*}
Hence, together with (\ref{1.2}), we have $\lim_{r\to l^-}
\lambda_{1,p}(B(q,r))=0$ for $1<p\leq2$ as $n=2$.

Now, assume that $n\geq3$. First, by the construction of $\phi_{m}$
above, we have for $1<p<3$
\begin{eqnarray*}
\int_{M}|\phi_{m}-1|^{p}\leq\int_{M\backslash{B_{R_m}}}1^p=|M|-V(R_m)\rightarrow0,
\qquad {\rm{as}}~~ m\rightarrow\infty.
\end{eqnarray*}
On the other hand, let $F(s)=\left(f(s)\right)^{n-1}$. Then, for
$n\geq3$, $F(l)=F'(l)=0$. By applying the Taylor's formula for $s$
close to $l$, we have $F(s)=F(l)+F'(l)(l-s)+\psi(s,l)(s-l)^2$, where
\begin{eqnarray*}
\psi(s,l)=\int_{0}^{1}(1-t)F''(l+t(s-l))dt.
\end{eqnarray*}
For a sufficiently small $\epsilon>0$, there exists a constant
$B_2>0$ such that $|\psi(s,l)|\leq{B_2}$ for $|l-s|<\epsilon$. Let
$R_m=l-\alpha^{m}$ with $0<\alpha<1$ a sufficiently small constant,
and then, for $1<p<3$, we have
\begin{eqnarray*}
\int_{M}|\nabla(\phi_{m}-1)|^{p}&\leq&\frac{V(R_{m+1})-V(R_m)}{\left(R_{m+1}-R_{m}\right)^{p}}=\frac{w_{n}}{\left(R_{m+1}-R_{m}\right)^{p}}
\int_{R_m}^{R_{m+1}}\psi(s,l)(s-l)^{2}ds\\
&\leq&\frac{w_{n}B_2}{\left(R_{m+1}-R_{m}\right)^{p}}
\int_{R_m}^{R_{m+1}}(s-l)^{2}ds=\frac{w_{n}B_2}{\left(\alpha^{m}-\alpha^{m+1}\right)^{p}}
\int_{\alpha^{m+1}}^{\alpha^{m}}s^{2}ds\\
&=&\frac{w_{n}B_{2}(1-\alpha^{3})}{3(1-\alpha)^{p}}\alpha^{m(3-p)}\rightarrow0,\qquad
{\rm{as}}~~m\rightarrow\infty.
\end{eqnarray*}
Hence, together with (\ref{1.2}), we have $\lim_{r\to l^-}
\lambda_{1,p}(B(q,r))=0$ for $1<p<3$ as $n\geq3$. Our proof is
finished.
\end{proof}

We also need the following notions, which can be found in
\cite{fmi,PhDJingMao}.
\begin{defn} \label{def2}
 Given a continuous function
$k:[0,l)\rightarrow \mathbb{R}$, we say that $M$ has  a radial Ricci
curvature lower bound $(n-1)k$ along any unit-speed minimizing
geodesic starting from a point $q\in{M}$ if
\begin{eqnarray}  \label{2.4}
{\rm{Ricci}}(v_x,v_x)\geq(n-1)k(r(x)), ~~\forall x\in
M\backslash{Cut(q)},
\end{eqnarray}
where ${\rm{Ricci}}$ is the Ricci curvature of $M$.
\end{defn}

\begin{defn} \label{def3}
Given a continuous function $k:[0,l)\rightarrow \mathbb{R}$, we say
that $M$ has a radial sectional curvature upper bound $k$ along any
unit-speed minimizing geodesic starting from a point $q\in{M}$ if
\begin{eqnarray} \label{2.5}
K(v_{x},V)\leq{k(r(x))},  ~~\forall x\in M\backslash{Cut(q)},
\end{eqnarray}
where $V\perp{v_{x}}$, $V\in{S^{n-1}_{x}}\subseteq{T_{x}M}$, and
$K(v_{x},V)$ is the sectional curvature of the plane spanned by
$v_x$ and $V$.
\end{defn}

\begin{remark} \rm{
As pointed out in remark 2.4 of \cite{fmi} or remark 2.1.5 of
\cite{PhDJingMao}, for $x=\gamma_{\xi}(t)$, since $r(x)=d(q,x)=t$
and $\frac{d}{dt}|_{x}=\nabla{r(x)}=v_x$, we know that the
inequalities (\ref{2.4}) and (\ref{2.5}) become
${\rm{Ricci}}(\frac{d}{dt},\frac{d}{d{t}}) \geq(n-1)k(t)$ and
$K(\frac{d}{dt},V)\leq{k(t)}$, respectively. Besides, for
convenience, if a manifold satisfies (\ref{2.4}) (resp.,
(\ref{2.5})), then we say that $M$ has \emph{a radial Ricci
curvature lower bound w.r.t. a point $q$} (resp., \emph{a radial
sectional curvature upper bound w.r.t. a point $q$}), that is to
say, its \emph{radial Ricci curvature is bounded from below w.r.t.
$q$} (resp., \emph{radial sectional curvature is bounded from above
w.r.t. $q$}).
 }
\end{remark}

For a prescribed $n$-dimensional complete manifold $M$, we would
like to construct the optimal continuous functions $k_{\pm}(q,t)$
w.r.t. a given base point $q\in{M}$, satisfying Definitions
\ref{def2} and \ref{def3}, respectively. We first recall that, for
$\xi\in{S}_{q}^{n-1}\subseteq{T_{q}M}$,
$\gamma_{\xi}(t)=\exp_{q}(t\xi)$ and its derivative
$\gamma'_{\xi}(t)$ are depending smoothly on the variables
$(t,\xi)$. Let
$\mathbb{D}_{q}:=\{(t,\xi)\in[0,\infty)\times{S_{q}^{n-1}}|0\leq
t<d_{\xi}\}$ with closure
$\overline{\mathbb{D}}_{q}=\{(t,\xi)\in[0,\infty)\times{S_{q}^{n-1}}|0\leq
t\leq d_{\xi}\}$. Then we can define
\begin{eqnarray} \label{extraII}
k_{-}(q,t):=\min\limits_{\{\xi|(t,\xi)\in\overline{\mathbb{D}}_{q}\}}
\frac{{\rm{Ricci}}_{\gamma_{\xi}(t)}\left(\frac{d}{dt}|_{\exp_{q}(t\xi)},\frac{d}{dt}|_{\exp_{q}(t\xi)}\right)}{n-1},
\qquad 0\leq t<l(q),
\end{eqnarray}
and
 \begin{eqnarray} \label{extraIII}
k_{+}(q,t):=\max_{\{(\xi,V)|
\gamma'_{\xi}(t)\perp{V}\}}K_{\gamma_{\xi}(t)}\left(\frac{d}{dt}\Big{|}_{\exp_{q}(t\xi)},V\right),
\qquad 0\leq t<inj(q).
\end{eqnarray}
If $l(q)<+\infty$, the above functions can be continuously extended
to $t=l(q)$ and $t=inj(q)$, respectively. Furthermore, if $M$ is
closed, the injectivity radius $inj(M):=\min_{q\in{M}}inj(q)$ of $M$
is a positive constant. Clearly, in this case $k_{\pm}(q,t)$ are
continuous, which can be obtained by applying the uniform continuity
of continuous functions on compact sets. Therefore, for a bounded
domain $\Omega\subseteq{M}$, one can always find optimally
continuous bounds $k_{\pm}(q,t)$ for the radial sectional and Ricci
curvatures w.r.t. some point $q\in\Omega$. This implies that the
assumptions on curvatures in Definitions \ref{def2} and \ref{def3}
are natural and advisable. Especially, when $M$ is a complete
surface, then $k_{\pm}(q,t)$ defined by (\ref{extraII}) and
(\ref{extraIII}) are actually the minimum and maximum of the
Gaussian curvature on geodesic circles centered at $q$ of radius $t$
on $M$.

Now, we would like to give explicit expressions of the radial
sectional and Ricci curvatures for any spherically symmetric
manifold. To this end, we should use some facts about the warped
product given in \cite{k,p}.

By proposition 42 and corollary 43 of chapter 7 in \cite{k} or
subsection 3.2.3 of chapter 3 in \cite{p}, we know that the radial
sectional curvature, and the radial component of the Ricci tensor of
the spherically symmetric manifold
$M^{\ast}=[0,l)\times_{f(t)}\mathbb{S}^{n-1}$ with the base point
$q$ are given by
\begin{equation} \label{radialcurvatures}
\begin{array}{ll}
K(V,\frac{d}{dt})=R(\frac{d}{d t},V,\frac{d}{d t}, V)
=-\frac{f''(t)}{f(t)}&\mbox{~for}~~ V\in T_{\xi}\mathbb{S}^{n-1},
~|V|_{g}=1,\\
{\rm{Ricci}}(\frac{d}{d t},\frac{d}{d t})=-(n-1)\frac{f''(t)}{f(t)}.
\end{array}
\end{equation}
Thus, Definition  \ref{def1} (resp., Definition {\ref{def2}}) is
satisfied with equality in (\ref{2.4}) (resp., (\ref{2.5})) and
$k(t)=-f''(t)/f(t)$. From (\ref{radialcurvatures}), we know that, in
order to define curvature tensor away from $q$, we need to require
$f\in{C}^{2}\left((0,l)\right)$. Furthermore, if $f''(0)=0$, and $f$
is $C^3$ at $t=0$, then we have $\lim_{t\rightarrow0}k(t)=-f'''(0)$.
Although $\nabla{r}$ is not defined at $x=q$, $k(t)$ is usually
required to be continuous at $t=0$, which is equal to require $f$ to
be $C^3$ at $t=0$. When $n=2$, $M^{\ast}$ is a surface, and if
$|f'(t)|\leq1$, then the mapping
\begin{eqnarray*}
\phi(t,\theta)=\left(f(t)\cos\theta,f(t)\sin\theta,h(t)\right),
\end{eqnarray*}
with $h(t)=\int_{0}^{t}\sqrt{1-(f'(t))^{2}}$, defines an isometric
embedding of $M^{\ast}$ into a surface of revolution in
$\mathbb{R}^{3}$. If the Gaussian curvature of $M^{\ast}$ is
negative at $q$, then no such local embedding exists near the base
point $q$, since $f'(t)>1$ near $t=0$ (see
(\ref{radialcurvatures})).

Define a function $\widetilde{\theta}(t,\xi)$ on
$M\backslash{Cut(q)}$ as follows
\begin{eqnarray} \label{functiontheta}
\widetilde{\theta}(t,\xi)=\left[\frac{J(t,\xi)}{f(t)}\right]^{n-1}.
\end{eqnarray}
Then we have the following \emph{generalized Bishop's volume
comparison} results, which correspond to theorem 3.3, corollary 3.4,
and theorem 4.2 in \cite{fmi} (equivalently, theorem 2.2.3,
corollary 2.2.4 and theorem 2.3.2 in \cite{PhDJingMao}).
\begin{theorem} (\cite{fmi,PhDJingMao}, generalized Bishop's volume comparison theorem I) \label{BishopI}
Given $\xi \in S_q^{n-1}\subseteq{T_{q}M}$, and a model space
$M^-=[0,l)\times_f \mathbb{S}^{n-1}$ w.r.t. $q^-$,
 under the curvature assumption on the radial Ricci tensor,
$\mathrm{Ricci}(\nu_x,\nu_x)\geq -(n-1)f''(t)/f(t)$ on $M$, for
$x=\gamma_{\xi}(t)=\exp_{q}(t\xi)$
 with $t<\min\{d_{\xi}, l\}$,  the function
$\widetilde{\theta}$ is nonincreasing in $t$. In particular, for all
$t<\min\{d_\xi, l\}$~ we have $J(t,\xi)\leq f(t)$. Furthermore, this
inequality is strict for all $t\in (t_0,t_1]$, with $0\leq
t_0<t_1<\min \{d_{\xi},l\}$,
 if the above  curvature
assumption holds with
 a strict inequality for $t$ in the same interval. Besides, we have
$$V(B(q,r_0))\leq V(V_n(q^{-},r_0)),$$
with equality if and only if $B(q,r_0)$ is isometric to
$V_n(q^{-},r_0)$.
\end{theorem}

\begin{theorem} (\cite{fmi,PhDJingMao}, generalized Bishop's volume comparison theorem II) \label{BishopII}
Assume $M$ has a radial sectional curvature upper bound
$k(t)=-\frac{f''(t)}{f(t)}$ w.r.t. $q\in{M}$ for
$t<\beta\leq\min\{inj_{c}(q),l\}$, where
$inj_{c}(q)=\inf_{\xi}c_{\xi}$, with $\gamma_{\xi}(c_{\xi})$  a
first conjugate point along the geodesic
$\gamma_{\xi}(t)=\exp_{q}(t\xi)$. Then on $(0,\beta)$
\begin{eqnarray}  \label{4.18}
\left(\frac{\sqrt{|g|}}{f^{n-1}}\right)'\geq0, \quad\quad
\sqrt{|g|}(t)\geq{f^{n-1}(t)},
 \end{eqnarray}
and  equality occurs in the first inequality at $t_{0}\in(0,\beta)$
if and only if
 \begin{eqnarray*}
\mathcal{R}=-\frac{f''(t)}{f(t)}, \quad \mathbb{A}=f(t)I,
 \end{eqnarray*}
 on all of $[0,t_{0}]$.
\end{theorem}

\section{A Cheng-type isoperimetric inequality for the p-Laplace operator}
\renewcommand{\thesection}{\arabic{section}}
\renewcommand{\theequation}{\thesection.\arabic{equation}}
\setcounter{equation}{0} \setcounter{maintheorem}{0}

We need the following proposition, which will be used in the proof
of Theorem \ref{theorem1} below.

\begin{proposition} \label{proposition1}
Let $T(t)$ be any solution of
\begin{eqnarray}  \label{3.1}
\left[|T'|^{p-2}f(t)^{n-1}T'\right]'+\lambda{f(t)^{n-1}}T|T|^{p-2}=0,
\qquad 1<p<\infty,
\end{eqnarray}
where $f(t)>0$ on the interval $(0,\beta)$. Then for $\Re=T'$ we
have that $\Re|(0,\beta]<0$ whenever we are given that
$T|(0,\beta)>0$, and $\lambda>0$.
\end{proposition}

\begin{proof}  Since $f(t)>0$ on the interval $(0,\beta)$,
and
\begin{eqnarray*}
|T'|^{p-2}f(t)^{n-1}T'(t)=-\lambda\int^{t}_{0}f(t)^{n-1}T|T|^{p-2}dt,
\end{eqnarray*}
the claim of the proposition follows.
\end{proof}

Denote by $B(q,r_{0})$ the open geodesic ball with center $q$ and
radius $r_{0}$ of an $n$-dimensional Riemannian manifold $M$ with
\emph{a radial Ricci curvature lower bound $(n-1)k(t)$ w.r.t. a
point $q\in{M}$}, and let $V_{n}(q^{-},r_{0})$ be the geodesic ball
with center $q^{-}$ and radius $r_{0}$ of an $n$-dimensional
spherically symmetric manifold $M^{-}$ with respect to the point
$q^{-}$ defined by $M^{-}:=[0,l)\times_{f(t)}\mathbb{S}^{n-1}$ with
$f(t)$ obtained by solving the initial value problem
\begin{eqnarray*}
\left\{
\begin{array}{lll}
-f''(t)=k(t)f(t), \qquad 0<t<r_{0}, \qquad f|(0,r_{0})>0,\\
f(0)=0,\\
f'(0)=1.  & \quad
\end{array}
\right.
\end{eqnarray*}
 We always assume $r_{0}<\min\{l(q),l\}$ with $l(q)$ defined in (\ref{important}).
In fact, we can prove the following.

\begin{theorem} \label{theorem1}
Suppose $M$ is a
 complete $n$-dimensional Riemannian manifold with a radial Ricci curvature
 lower bound $(n-1)k(t)=-\frac{(n-1)f''(t)}{f(t)}$ w.r.t. a point $q$,
 and  $M^{-}$ is an $n$-dimensional spherically symmetric manifold with respect to
a point $q^{-}$ whose metric is given by (\ref{2.3}).
  Then, for $1<p<\infty$, we have
\begin{eqnarray} \label{3.2}
\lambda_{1,p}(B(q,r_{0}))\leq\lambda_{1,p}(V_{n}(q^{-},r_{0})),
\end{eqnarray}
where $\lambda_{1,p}(\cdot)$ denotes the first Dirichlet eigenvalue
of the $p$-Laplacian of the corresponding geodesic ball. Moreover,
the equality holds if and only if $B(q,r_{0})$ is isometric to
$V_{n}(q^{-},r_{0})$.
\end{theorem}

\begin{proof}
Let $\phi$ be the nonnegative eigenfunction of the first eigenvalue
of the Dirichlet $p$-Laplacian on $V_{n}(q^-,r_{0})$. By (\ref{1.1})
and (\ref{2.3}), the $p$-Laplacian on the spherically symmetric
manifold $M^{-}$ under the geodesic polar coordinates at $q^{-}$ is
given by
\begin{eqnarray*}
\triangle_{p}=|\nabla(\cdot)|^{p-2}\frac{d^2}{dt^{2}}+\frac{d}{dt}\left(|\nabla(\cdot)|^{p-2}\right)\frac{d}{dt}+(n-1)\frac{f'(t)}{f(t)}|\nabla(\cdot)|^{p-2}\frac{d}{dt}+
\frac{1}{f^{2}(t)}\triangle_{p,\mathbb{S}^{n-1}},
\end{eqnarray*}
where $\triangle_{p,\mathbb{S}^{n-1}}$ denotes the $p$-Laplacian on
the $(n-1)$-dimensional unit sphere $\mathbb{S}^{n-1}$. Then the
eigenfunction $\phi$ should be a radial function satisfying
\begin{eqnarray} \label{3.3}
(p-1)|\phi'(t)|^{p-2}\phi''(t)+(n-1)\frac{f'(t)}{f(t)}|\phi'(t)|^{p-2}\phi'(t)+\lambda_{1,p}\left(V_{n}(q^-,r_0)\right)|\phi(t)|^{p-2}\phi(t)=0
\end{eqnarray}
and the boundary conditions $\phi(r_{0})=0$, $\phi'(0)=0$. Clearly,
(\ref{3.3}) has the form of (\ref{3.1}).

Let $r$ be the distance to the point $q$ on $M$, and then
$\phi\circ{r}$ vanishes on the boundary $\partial{B}(q,r_{0})$.
Hence, by (\ref{1.2}), we obtain
\begin{eqnarray*}
\lambda_{1,p}(B(q,r_{0}))\leq\frac{\int|d\phi\circ{r}|^p}{\int|\phi\circ{r}|^{p}},
\end{eqnarray*}
 where we drop $B(q,r_{0})$ and volume element $dB(q,r_{0})$ for the above expression. Let
$a(\xi):=\min\{d_{\xi}$, $r_{0}\}$. Then, clearly,
$a(\xi)\leq{r_{0}}$ and $\exp_{q}(d_{\xi}\cdot\xi)$ is the cut-point
of $q$ along the geodesic $\gamma_{\xi}(t)=\exp_{q}(t\xi)$. Under
the geodesic polar coordinates $(t,\xi)$ around $q$ ,we have
\begin{eqnarray*}
\int\limits_{B(q,r_{0})}|d\phi\circ{r}|^p=\int\limits_{\xi\in\mathbb{S}^{n-1}}\left[\int\limits^{a(\xi)}_{0}
|\phi'(t)|^{p}\times{f^{n-1}(t)}\times\theta(t\xi)dt\right]d\sigma,
\end{eqnarray*}
\begin{eqnarray*}
\int\limits_{B(q,r_{0})}|\phi\circ{r}|^{p}=\int\limits_{\xi\in\mathbb{S}^{n-1}}\left[\int\limits^{a(\xi)}_{0}
|\phi(t)|^{p}\times{f^{n-1}(t)}\times\theta(t\xi)dt\right]d\sigma,
\end{eqnarray*}
where $d\sigma$ is the canonical measure of
$\mathbb{S}^{n-1}\equiv{S_{q}^{n-1}}$, and
$\theta(t\xi):=\sqrt{\det(g_{ij})}\times{f^{1-n}(t)}$.

On the other hand, since $f(t)>0$, $\phi\geq0$ for $0<t<r_{0}$, by
Proposition \ref{proposition1} we have $\phi'(t)\leq0$ for
$0<t<r_{0}$. By straightforward computation, it follows that
\begin{eqnarray}  \label{3.4}
&&\int\limits^{a(\xi)}_{0}
|\phi'(t)|^{p}\times{f^{n-1}(t)}\times\theta(t\xi)dt=-\phi|\phi'(t)|^{p-1}f^{n-1}(t)\theta(t\xi)|^{a(\xi)}_{0}+
\int\limits^{a(\xi)}_{0}\frac{\phi}{f^{n-1}(t)\theta(t\xi)}\cdot\nonumber\\
\lefteqn{ \qquad \qquad
\frac{d}{dt}\left[f^{n-1}(t)\theta(t\xi)|\phi'(t)|^{p-1}\right]f^{n-1}(t)\theta(t\xi)dt,}
\end{eqnarray}
\begin{eqnarray}  \label{3.5}
\lefteqn{\frac{1}{f(t)^{n-1}\theta(t\xi)}\frac{d}{dt}\left[f(t)^{n-1}\theta(t\xi)|\phi'(t)|^{p-1}\right]=}\nonumber \\
&&-|\phi'(t)|^{p-2}\left\{(p-1)\phi''(t)
+\left[\frac{(n-1)f'(t)}{f(t)}+\frac{\frac{d\theta(t\xi)}{dt}}{\theta(t\xi)}\right]\phi'(t)\right\}.
\end{eqnarray}
By (\ref{J}), we have
$\theta(t\xi)=\sqrt{det(g_{ij})}\times{f^{1-n}(t)}=J^{n-1}{f^{1-n}(t)}$,
which coincides with the function $\widetilde{\theta}$ defined in
(\ref{functiontheta}). Substituting this to (\ref{3.5}) results in
\begin{eqnarray}\label{3.6}
\frac{1}{f(t)^{n-1}\theta(t\xi)}\frac{d}{dt}\left[f^{n-1}(t)\theta(t\xi)|\phi'(t)|^{p-1}\right]=-|\phi'(t)|^{p-2}\cdot\nonumber\\
\left\{(p-1)\phi''(t)+\left[\frac{(n-1)f'(t)}{f(t)}+(n-1)\frac{f(t)}{J(t)}\left(\frac{J(t)}{f(t)}\right)'\right]\phi'(t)\right\}.
\end{eqnarray}
 Since $M$ has \emph{a radial Ricci curvature lower bound
$(n-1)k(t)=-(n-1)f''(t)/f(t)$ w.r.t. the point $q$}, then by Theorem
\ref{BishopI}, (\ref{J0}) and the fact $f(0)=0$, $f'(0)=1$, we have
\begin{eqnarray}  \label{3.7}
\left(\frac{J}{f}\right)'\leq0
\end{eqnarray}
for $0<t<r_{0}$.

Therefore, by (\ref{3.3}), (\ref{3.6}), (\ref{3.7}) and the
nonpositivity of $\phi'(t)$ on $(0,r_{0})$,  we have
\begin{eqnarray}  \label{3.8}
&&\int\limits^{a(\xi)}_{0}\frac{\phi}{f^{n-1}(t)\theta(t\xi)}\frac{d}{dt}\left[f^{n-1}(t)\theta(t\xi)
|\phi'(t)|^{p-1}\right]f^{n-1}(t)\cdot\theta(t\xi)dt\leq\int\limits^{a(\xi)}_{0}|\phi|^{p}\lambda_{1}(V_{n}(q^{-},r_{0}))\nonumber\\
\lefteqn{\qquad \qquad\cdot{f^{n-1}(t)}\theta(t\xi)dt.}
\end{eqnarray}
Substituting (\ref{3.8}) into (\ref{3.4}) yields
\begin{eqnarray} \label{3.9}
&&\int\limits^{a(\xi)}_{0}
|\phi'(t)|^{p}\times{f(t)^{n-1}}\times\theta(t\xi)dt\leq-\phi(a(\xi))|\phi'(a(\xi))|^{p-1}f^{n-1}(a(\xi))\cdot\theta(a(\xi)\xi)\nonumber\\
\lefteqn{\qquad \qquad \quad
+\int\limits^{a(\xi)}_{0}|\phi|^{p}\lambda_{1,p}\left(V_{n}(q^{-},r_{0})\right)f^{n-1}(t)\cdot\theta(t\xi)dt.}
\end{eqnarray}
Recall that $\phi\geq0$, and then from (\ref{3.9}) we have
\begin{eqnarray*}
\int\limits^{a(\xi)}_{0}
|\phi'(t)|^{p}\times{f(t)^{n-1}}\times\theta(t\xi)dt\leq
\int\limits^{a(\xi)}_{0}|\phi|^{p}\lambda_{1,p}(V_{n}(q^{-},r_{0}))f(t)^{n-1}\theta(t\xi)dt,
\end{eqnarray*}
and furthermore,
$$
\int\limits_{\xi\in{\mathbb{S}}^{n-1}}\left[\int\limits^{a(\xi)}_{0}
|\phi'(t)|^{p}\times{f^{n-1}(t)}\times\theta(t\xi)dt\right]d\sigma\leq
$$
 \begin{eqnarray*}
\int\limits_{\xi\in{\mathbb{S}}^{n-1}}\left[\int\limits^{a(\xi)}_{0}\lambda_{1,p}(V_{n}(q^{-},r_{0}))
|\phi|^{p}\times{f^{n-1}(t)}\times\theta(t\xi)dt\right]d\sigma,
\end{eqnarray*}
which implies
$\lambda_{1,p}(B(q,r_{0}))\leq\lambda_{1,p}(V_{n}(q^{-},r_{0}))$.

When equality holds, we have that $a(\xi)=r_{0}$ for almost all
$\xi\in{S}^{n-1}_{q}$. Hence $a(\xi)\equiv{r_{0}}$ for all $\xi$. We
can then conclude that $J(t,\xi)=f(t)$, and by Theorem
\ref{BishopI}, we know that $B(q,r_{0})$ is isometric to
$V_{n}(q^{-},r_{0})$. This completes the proof of Theorem
\ref{theorem1}.
\end{proof}

\begin{remark} \label{remark1}
\rm{We would like to point out the following facts about Theorem
\ref{theorem1}.

(1) Theorem \ref{theorem1} is sharper than theorem 1.1 in \cite{am}
or theorem 3 in \cite{h}. In fact, if an $n$-dimensional complete
Riemannian manifold $M$ has \emph{a radial Ricci curvature lower
bound $(n-1)k(t)$ w.r.t. a point $q\in{M}$}, where $k(t)$ is a
continuous function on the interval $[0,r_{0})$, and let
$k_{0}:=\inf_{0\leq{t}<{r_{0}}}k(t)$, then by Theorem \ref{theorem1}
we have
\begin{eqnarray*}
\lambda_{1,p}\left(B(q,r_{0})\right)\leq\lambda_{1,p}\left(V_{n}(q^{-},r_{0})\right)\leq
\lambda_{1,p}\left(V_{n}(k_{0},r_{0})\right),
\end{eqnarray*}
where $V_{n}(k_{0},r_{0})$ is a geodesic ball with radius $r_{0}$ in
the $n$-dimensional space form with constant curvature $k_{0}$, and
the other symbols have the same meanings as those in Theorem
\ref{theorem1}. However, by theorem 1.1 in \cite{am} or theorem 3 in
\cite{h}, one can only have
\begin{eqnarray*}
\lambda_{1,p}\left(B(q,r_{0})\right)\leq\lambda_{1,p}\left(V_{n}(k_{0},r_{0})\right).
\end{eqnarray*}
We will show this fact clearly by Example \ref{example} of the next
section.

(2) Our comparison result (\ref{3.2}) is valid regardless of the
cut-locus, since the Lebesgue measure of the cut-locus is 0 with
respect to the $n$-dimensional Lebesgue measure of the manifold $M$,
which implies that integrations over the cut-locus vanish.
 }
\end{remark}

\begin{corollary} \label{strange}
Under the curvature conditions of the previous theorem, holding for
all $t<l(q)=l$ where $M^-=[0,l)\times_f \mathbb{S}^{n-1}$, if $M$ is
closed and  $M^-$ also closes i.e. $f(l)=0$ and  satisfies the
conditions in  Lemma \ref{closedmodel},
 then  for all $\xi$, $\exp_q(l\xi)$
is a  conjugate point of $q$,
 and $\lim_{r\to l^-} \lambda_{1,p}(B(q,r))=0$ with $1<p\leq2$ in case $n=2$, or
 $\lim_{r\to l^-} \lambda_{1,p}(B(q,r))=0$ with $1<p<3$ in case $n\geq3$.
\end{corollary}
\begin{proof}
The latter conclusion follows from Theorem \ref{theorem1} and Lemma
\ref{closedmodel}. Moreover, by Theorem \ref{BishopI}, we have
$J(l,\xi)=0$ for all $\xi\in{S^{n-1}_{q}}$, which implies that
$\exp_{q}(l\xi)$ is a conjugate point of $q$.
\end{proof}

\section{Estimates for the first eigenvalue of the $p$-Laplacian}
\renewcommand{\thesection}{\arabic{section}}
\renewcommand{\theequation}{\thesection.\arabic{equation}}
\setcounter{equation}{0} \setcounter{maintheorem}{0}

In this section, we would like to use Theorem \ref{theorem1} and
some other existing estimates to get bounds for the first eigenvalue
of the $p$-Laplacian of geodesic balls on a Riemannian manifold with
\emph{radial Ricci curvature bounded from below w.r.t. some point}.
Before that, we need the following concept.
\begin{defn} \label{def4}
The Cheeger constant $h(\Omega)$ of a domain $\Omega$ (with
boundary) is defined to be
\begin{eqnarray*}
h(\Omega):=\inf\limits_{\Omega'}\frac{\rm{vol}(\partial\Omega')}{\rm{vol}(\Omega')},
\end{eqnarray*}
where $\Omega'$ ranges over all open submanifolds of $\Omega$ with
compact closure in $\Omega$ and smooth boundary $\partial\Omega'$,
and $\rm{vol}(\partial\Omega')$ and $\rm{vol}(\Omega')$ denote the
volumes of $\partial\Omega'$ and $\Omega'$ respectively.
\end{defn}

\begin{theorem} (\cite{ld,h}) \label{theorem2}
 For any bounded domain $\Omega$ with
piecewise smooth boundary in a complete Riemannian manifold, we have
\begin{eqnarray*}
\lambda_{1,p}(\Omega)\geq\left(\frac{h(\Omega)}{p}\right)^p.
\end{eqnarray*}
\end{theorem}

Let $D$ vary over all smooth subdomains of $\Omega$
 whose boundary $\partial{D}$ does not
touch $\partial{\Omega}$, and define the Cheeger quotient of $D$ as
$Q(D):={\rm{vol}}(\partial{D})/{\rm{vol}}(D)$. We call a subset $w$
of $\Omega$ a Cheeger domain of $\Omega$ if $Q(w)=h(\Omega)$. The
existence, (non)uniqueness and regularity of Cheeger domains are
interesting and important topics in Differential Geometry, but here
we do not want to focus on them. Generally, it is difficult to get
the Cheeger domain for a prescribed domain on a general Riemannian
manifold. But for some special cases, it is not difficult. For
instance, the Cheeger domain $w$ for a unit square
$S_{1}\subseteq{\mathbb{R}^2}$ is a square with its corners rounded
off by circular arcs of radius $\rho=(4-2\sqrt{\pi})/(4-\pi)$, which
has been pointed out in \cite{bv}. Especially, for a ball $B_{R}$
with radius $R$ in the Euclidean $n$-space $\mathbb{R}^n$, its
Cheeger domain coincides with itself, which implies that its Cheeger
constant is $h(B_{R})=n/R$.

In \cite{grig99}, Grigor'yan has obtained estimates for the
so-called \emph{principal $p$-frequency} ($1<p<\infty$) of geodesic
balls on spherically symmetric manifolds. The principal
$p$-frequency there is actually the first eigenvalue of the
$p$-Laplacian. More precisely, if $B_{R}=V_{n}(q^-,R)$ be a geodesic
ball centered at the point $q^-$ with radius $R$ on the prescribed
$n$-dimensional spherically symmetric manifold $M^{-}$ with the
metric (\ref{2.3}), then the first eigenvalue $\lambda_{1,p}(B_{R})$
of the $p$-Laplacian of this geodesic ball satisfies
\begin{eqnarray} \label{4.1}
a_{p}m_{p}(B_{R})\leq\lambda_{1,p}(B_{R})\leq{m_{p}(B_{R})},
\end{eqnarray}
where $m_{p}(B_{R})$ and $a_{p}$ are given by
\begin{eqnarray*}
m_{p}(B_{R})=\frac{1}{\sup\limits_{r\leq{R}}\left\{\int_{0}^{r}f(t)^{n-1}dt\left[\int_{r}^{R}f(t)^{\frac{1-n}{p-1}}dt\right]^{p-1}\right\}},
\end{eqnarray*}
and
\begin{eqnarray*}
a_{p}=\left\{
\begin{array}{ll}
{(p-1)^{p-1}p^{-p}, \quad if \quad p>1,}\\
 1, \quad\quad\quad\quad\quad\quad~ if \quad p=1,
\end{array}
\right.
 \end{eqnarray*}
respectively (cf. sections 2 and 7 in \cite{grig99}).

Hence, by Theorem \ref{theorem1}, Theorem \ref{theorem2} and
(\ref{4.1}), we have the following estimates.
\begin{theorem} \label{theorem3}
Let $M$ be a
 complete $n$-dimensional Riemannian manifold with a radial Ricci curvature
 lower bound $(n-1)k(t)=-\frac{(n-1)f''(t)}{f(t)}$ w.r.t. $q\in{M}$. Then, for any
 $1<p<\infty$,
 the first Dirichlet eigenvalue $\lambda_{1,p}(B(q,R))$
 of the $p$-Laplacian of the geodesic ball $B(q,R)$ on $M$ satisfies
 \begin{eqnarray} \label{4.2}
\left(\frac{h(B(q,R))}{p}\right)^{p}\leq\lambda_{1,p}(B(q,R))\leq{m_{p}(B_{R})},
 \end{eqnarray}
 where
 $h(B(q,R))$ is the Cheeger constant of $B(q,R)$, and $m_{p}(B_{R})$ is defined in (\ref{4.1}).
 Especially, when $M=\mathbb{R}^{n}$, we have
 \begin{eqnarray} \label{4.3}
\left(\frac{n}{Rp}\right)^{p}\leq\lambda_{1,p}\left(B(R)\right)\leq{C}(n,p,R)
 \end{eqnarray}
 for any ball $B(R)\subseteq{\mathbb{R}^n}$ with radius $R$, where $C(n,p,R)$ is given by
\begin{eqnarray*}
C(n,p,R)=\left\{
\begin{array}{lll}
\frac{p^{\frac{p^2-p}{p-n}}}{n^{\frac{np-p}{p-n}}\cdot(p-1)^{p-1}\cdot{R^{p}}},
\qquad n\neq{p},
\\
\\
\frac{n^{n}e^{n-1}}{(n-1)^{n-1}\cdot{R}^{n}}, \qquad\quad\quad~ n=p.
\end{array}
\right.
\end{eqnarray*}
\end{theorem}

Here we would like to use an example given in \cite{PhDJingMao} to
show that our Theorem \ref{theorem3} is useful.

\begin{example} \label{example}
\rm{In general, it is difficult to get the Cheeger constant of a
geodesic ball on a curved manifold. So, for a Riemannian manifold
with \emph{a radial Ricci curvature lower bound w.r.t. some point},
(\ref{4.2}) may not give us any interesting information on the lower
bound for the first eigenvalue of the $p$-Laplacian, while it can
give us an upper bound numerically by using Mathematica.

Denote by $E^3$ the $3$-dimensional Euclidean space with a Cartesian
coordinate system $\left\{x,y,z\right\}$ with the origin $o$. Now,
consider a circle $\mathcal {C}$ in the $xoy$-plane given by
$(x-1)^2+y^2=1/4$, and then rotating it w.r.t. the $y$-axis results
in a ring torus $\mathcal {T}$ with the major radius 1 and the minor
radius 0.5. Of course, we can parameterize the torus $\mathcal{T}$
in $E^3$ by
\begin{eqnarray*}
\left\{
\begin{array}{lll}
x=(1+0.5\cos v)\cos u,\\
y=0.5\sin v, \\
z=(1+0.5\cos v)\sin u, &\quad
\end{array}
\right.
\end{eqnarray*}
with $u,v\in[0,2\pi)$. So, the Gaussian curvature of $\mathcal{T}$
is given by
 \begin{eqnarray*}
 K=\frac{4\cos v}{2+\cos v}, \qquad v\in[0,2\pi).
 \end{eqnarray*}

Now, we want to use our estimates (\ref{4.2}) to give an upper bound
for the first eigenvalue of the $p$-Laplacian on a geodesic ball
$B(q,\delta)$ with radius $\delta$ and center $q\in\mathcal{T}$.
Here we choose $0<\delta<\pi/2$, otherwise the geodesic ball will
overlap. According to the position of the point $q$, we divide into
three cases to derive the upper bound here.

 Case (I): If $q$ is one of those points which are farthest from
the $y$-axis, that is, $q$ locates on the circle $C_{1}$ in
$xoz$-plane defined by $x^2+z^2=9/4$. Without loss of generality, we
can choose $q$ to be the point $(3/2,0,0)$, which implies that $q$
is also on the circle $\mathcal{C}$.

In this case, the parameter $v$ satisfies $v=0$ at $q$. Define a
function $k(v):=4\cos v/(2+\cos v)$, which is decreasing on the
interval $[0,\pi]$ and increasing on the interval $(\pi,2\pi)$.
Clearly, $k(v)$ attains its minimum $k_{min}=-4$ at $v=\pi$. At the
point $(1/2,0,0)$ of the circle $\mathcal{C}$, the parameter $v$
attains value $\pi$. We know that the two arcs of $\mathcal{C}$
starting from $q$ are two geodesics of $\mathcal{T}$, and if we move
away from $q$ on $\mathcal{T}$ with a distance $t$ ($0<t<\pi/2$),
the angle parameter $v$ increases or decreases most quickly, with a
quantity $2t$, along these two arcs. Therefore, for the function
$k(v)$ defined above, together with its monotonicity on the interval
$[0,2\pi)$, we have the Gaussian curvature $K$ satisfies
\begin{eqnarray} \label{4.4}
K\geq\frac{4\cos 2t}{2+\cos 2t},
\end{eqnarray}
where $t=d(q,\cdot)$ is the distance to $q$ on $\mathcal{T}$. This
implies that the best sectional curvature lower bound
$K_{lower}^{1}(t)$ can be chosen to be $K_{lower}^{1}(t)=4\cos
2t/(2+\cos 2t)$.

Case (II): If $q$ is one of those points which are nearest to the
$y$-axis, that is, $q$ locates on the circle $C_{2}$ in $xoz$-plane
defined by $x^2+z^2=1/4$. Without loss of generality, we can choose
$q$ to be the point $(1/2,0,0)$, which implies $q\in\mathcal{C}$.

In this case, by using a similar method as in Case (I), the Gaussian
curvature $K$ satisfies
\begin{eqnarray} \label{4.5}
K\geq-4,
\end{eqnarray}
which implies that the best sectional curvature lower bound
$K_{lower}^{2}(t)$ can be chosen to be $K_{lower}^{2}(t)=-4$.

Case (III): If $q$ is neither a point on the circle $C_{1}$ nor a
point on the circle $C_{2}$. Without loss of generality, we can
choose $q$ to be a point, which is different from the points
$(3/2,0,0)$ and $(1/2,0,0)$, on the circle $\mathcal{C}$.

Assume $v=\alpha$ at $q$ with $0<\alpha<\pi$ or $\pi<\alpha<2\pi$.
By the symmetry of $\mathcal{T}$ w.r.t. the $xoy$-plane, without
loss of generality, we can assume $0<\alpha<\pi$. In this case, by
using a similar method as in Case (I), the Gaussian curvature $K$
satisfies
\begin{eqnarray} \label{4.6}
K\geq\left\{
 \begin{array}{lll}
\frac{4\cos(\alpha+2t)}{2+\cos(\alpha+2t)}, \qquad
0\leq{t}\leq\frac{\pi-\alpha}{2},\\
\\
-4, \qquad\qquad\quad \frac{\pi-\alpha}{2}<t<\frac{\pi}{2},
 \end{array}
 \right.
\end{eqnarray}
which implies the sectional curvature lower bound $K_{lower}^{3}(t)$
can be chosen to be
\begin{eqnarray*} K_{lower}^{3}(t)=\left\{
 \begin{array}{lll}
\frac{4\cos(\alpha+2t)}{2+\cos(\alpha+2t)}, \qquad
0\leq{t}\leq\frac{\pi-\alpha}{2},\\
\\
-4, \qquad\qquad\quad \frac{\pi-\alpha}{2}<t<\frac{\pi}{2},
 \end{array}
 \right.
\end{eqnarray*}

 Correspondingly,
by using Mathematica to solve the initial value problem
\begin{eqnarray*}
\left\{
\begin{array}{lll}
-\frac{f''_{i}(t)}{f_{i}(t)}=K_{lower}^{i}(t), \qquad 0\leq{t}<\frac{\pi}{2},\\
f_{i}(0)=0,\\
f'_{i}(0)=1,  \qquad\qquad\quad~ i=1,2,3, & \quad
\end{array}
\right.
\end{eqnarray*}
with, without loss of generality, choosing $\alpha=\pi/2$ for
$K_{lower}^{3}(t)$, we can get $f_{i}(t)$ numerically for the above
three cases, and then the upper bounds for the first eigenvalue
follow easily (see Table 1 below). Actually, one could get the
graphs of $f_{1}(t)$, $f_{2}(t)$, and $f_{3}(t)$ as Figure 1 below.

\begin{figure*}[h!]
\centering
\includegraphics[height=6cm,width=8cm]{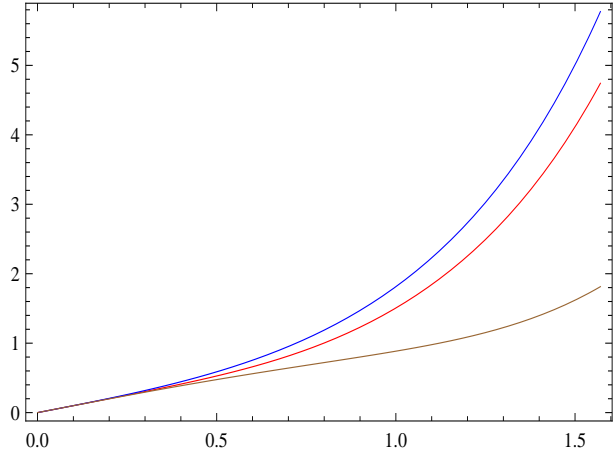}
\caption{Graphs of $f_{i}(t)$; the lowest one (brown) is $f_{1}(t)$
while the highest one (blue) is $f_{2}(t)$, and the middle one (red)
is $f_{3}(t)$.}
\end{figure*}

Correspondingly, the model surfaces for the geodesic ball
$B(q,r_{0})$ in the above three cases can be chosen to be
$M^{-}_{i}:=[0,r_{0})\times_{f_{i}(t)}\mathbb{S}^{n-1}$ ($i=1,2,3$).
Since $K^{2}_{lower}(t)\leq K^{3}_{lower}(t)\leq K^{1}_{lower}(t)$
for $0\leq t<r_{0}$, then by the Sturm-Picone comparison theorem, we
know that $f_{2}(t)\leq f_{3}(t)\leq f_{1}(t)$ for $0\leq t<r_{0}$
(see also Figure 1). As we have pointed out in Section
\ref{section2}, if the Gaussian curvature is nonnegative around
$q\in\mathcal{T}$, then the model surface could be locally embedded
into a surface of revolution in $\mathbb{R}^3$. So, here we could
only get a picture for $M_{1}^{-}$ by using Mathematica. One can see
Figure 2 in \cite{fmi} (equivalently, Figure 2.3 in
\cite{PhDJingMao}) for the graph of $M_{1}^{-}$. When $f'(t)$ starts
to be greater than $1$ for some $t=t_{0}$, the model surface stops
being isometrically embeddedable in $\mathbb{R}^3$, which implies
that its picture can not be drawn when $t\geq t_{0}$. We call this
$t_{0}$ ``\emph{stopping time}". The ``\emph{stopping time}" $t_{0}$
for our model surface $M_{1}^{-}$ here is $t_{0}\approx1.097$ (cf.
example 6.1 in \cite{fmi} or example 2.5.1 in \cite{PhDJingMao}).
For more information about the properties of the model manifolds of
prescribed manifolds, one could see \cite{fmi,PhDJingMao} in detail.




Without loss of generality, we can choose $\alpha=\pi/2$ in  Case
(III). Denote the upper bounds of the first Dirichlet eigenvalue of
the $p$-Laplacian in the above three cases by JM1, JM2 and JM3,
respectively. Then, for different $p$ and $\delta$, we have the
Table 1 below.

Table 1 makes sense, since it is difficult to compute the first
Dirichlet eigenvalue of the $p$-Laplacian on a geodesic ball of
$\mathcal{T}$, but, this table supplies us a range for the first
eigenvalue.

For Case (I) and Case (III), the lower bounds of the Gaussian
curvature w.r.t. the base point $q\in\mathcal{T}$ are given by
continuous functions of the distance parameter $t$, which are not
constant functions. By (1) of Remark \ref{remark1}, we know that if
we apply Theorem \ref{theorem1}, then the corresponding estimates
for the first eigenvalue of the $p$-Laplacian will be sharper than
the estimates obtained by using theorem 1.1 in \cite{am} or theorem
3 in \cite{h}. Of course, one may also use other examples about
elliptic paraboloid and saddle shown in \cite{fmi} to show the
advantage of our Theorem \ref{theorem1}, but, this example about
torus is enough.

In addition, for given $n$, $p$ and $R$, estimates (\ref{4.3}) give
an interval where the first Dirichlet eigenvalue of the
$p$-Laplacian on the ball $B(R)\subseteq\mathbb{R}^{n}$ locates.
Although, in \cite{rge}, the authors there have shown that one can
get the approximate value of the first eigenvalue of the
$p$-Laplacian of the ball $B(R)$ in the Euclidean space via the
inverse power method, we still think (\ref{4.3}) is useful, since it
can be used to check the validity of this approximate value of the
first eigenvalue at the first glance.

\vskip 1 mm
 {\bf Table 1} \quad Numerical values of the upper bounds of the first Dirichlet
 eigenvalue of the $p$-Laplacian
 \vskip 2mm
\begin{tabular}{|c |c| c c c c c c| }
\hline
&  & $\delta=\frac{\pi}{24}$ & $\delta=\frac{\pi}{12}$ & $\delta=\frac{\pi}{6}$ & $\delta=\frac{\pi}{4}$& $\delta=\frac{\pi}{3}$& $\delta=\frac{5\pi}{12}$ \\
\hline
      & $p=1.1$  &27.1285   & 12.5875  & 5.76216 & 3.615235 & 2.63716 & 2.18278  \\
      & $p=1.5$  & 129.804 & 45.6551 & 15.8426  & 8.43068 & 5.41996 & 3.98597  \\
JM1   & $p=2$    & 633.49 & 157.585 & 38.6834  &16.7921 & 9.29658 & 6.02468  \\
      & $p=2.5$ & 2643.65 &465.081 & 80.7606  & 28.6185 & 13.6868 & 7.87571  \\
      & $p=2.9$ & 7788.71 & 1038.53 & 136.711  & 41.1932 & 17.5401 & 9.19918  \\
\hline
      & $p=1.1$  & 27.3318   & 12.9637  & 6.43987 & 4.52941 & 3.69959 & 3.27638   \\
      & $p=1.5$  & 130.731 & 46.9574 & 17.6385  & 10.5314 & 7.67446 & 6.24665\\
JM2   & $p=2$    & 637.815 & 161.89 & 42.9072  & 20.8735 & 13.1648 & 9.60296  \\
      & $p=2.5$ & 2661.1 & 477.379 & 89.3207  & 35.4141 & 19.3209 & 12.609  \\
      & $p=2.9$ & 7839.06 & 1065.42 & 150.92  & 50.8147 & 24.6856 & 14.7308  \\
\hline
      & $p=1.1$  & 27.1916   & 12.7303  &6.13046 & 4.27308 & 3.53423 & 3.17492   \\
      & $p=1.5$  & 130.086 & 46.1295 & 16.7496 & 9.8136 & 7.1877 & 5.93221  \\
JM3   & $p=2$    & 634.785 & 159.108 & 40.7077  & 19.3037 & 12.1299 & 8.9332 \\
      & $p=2.5$ & 2648.82 & 469.358 & 84.735  & 32.6294 & 17.6292 & 11.5564  \\
      & $p=2.9$ & 7803.57 & 1047.8 & 143.195  & 46.7425 & 22.4114& 13.3865  \\
\hline
\end{tabular}

\vskip 2 mm
 }
\end{example}

\section{Some facts about the heat eqaution}
\renewcommand{\thesection}{\arabic{section}}
\renewcommand{\theequation}{\thesection.\arabic{equation}}
\setcounter{equation}{0} \setcounter{maintheorem}{0}

If we want to get the existence, or even give an explicit
expression, of the solution for the heat equation
(\ref{heatequation}) with a prescribed initial condition or
(Dirichlet or Neumann) boundary condition, we need to use a tool
named \emph{heat kernel}.

\begin{defn} \label{def1chapter1} \label{def4}
A fundamental solution, which is called the heat kernel, of the heat
equation on a prescribed Riemannian manifold $M$ is a continuous
function $H(x,y,t)$, defined on $M\times{M}\times(0,\infty)$, which
is $C^{2}$ with respect to $x$, $C^{1}$ with respect to $t$, and
which satisfies
\begin{eqnarray*}
L_{x}p=0, \qquad \lim\limits_{t\rightarrow0}H(x,y,t)=\delta_{y}(x),
\end{eqnarray*}
where $\delta_{y}(x)$ is the Dirac delta function, that is, for all
bounded continuous function $f$ on $M$, we have, for every
$y\in{M}$,
 \begin{eqnarray*}
 \lim\limits_{t\rightarrow0}\int_{M}H(x,y,t)f(x)dV(x)=f(y).
 \end{eqnarray*}
\end{defn}

By constructing a parametrix, the existence of the heat kernel on
compact or complete Riemannian manifolds, or even manifolds with
boundaries subject to either Dirichelt or Neumann boundary
conditions can be obtained (see, for instance, \cite{i}). In fact,
for a complete Riemannian manifold, one can have the following.

\begin{theorem} (\cite{rs}) \label{theoremheatchapter1}
Let $M$ be a complete Riemannian manifold, then there exists a heat
kernel $H(x,y,t)\in{C^{\infty}(M\times{M}\times\mathbb{R}^{+})}$
such that $\\$ (I) $H(x,y,t)=H(y,x,t)$,
 $\\$ (II) $\lim\limits_{t\rightarrow0}H(x,y,t)=\delta_{x}(y)$,
 $\\$ (III) $\left(\Delta-\frac{d}{dt}\right)H=0$,
 $\\$ (IV) $H(x,y,t)=\int_{M}H(x,z,t-s)H(z,y,s)dV(z)$.
\end{theorem}

In the next section, we would like to focus on the heat kernels of
geodesic balls on complete manifolds, and successfully obtain a
comparison result, which can be seen as an extension of
Debiard-Gaveau-Mazet's comparison result in \cite{dge} and
Cheeger-Yau's comparison result in \cite{cy}, for the heat kernel
with a Dirichlet or Neumann boundary condition -- see Theorem
\ref{7theoremchapter2} for the precise statement. There is a
connection between the heat kernel and the eigenvalues of the
Laplace operator. One can get a glance about this relation from the
following conclusion (cf. \cite{i}, p. 169).

\begin{theorem} (The Sturm-Liouville decomposition for the Dirichlet eigenvalue problem) \label{theoremheat2chapter1}
Given a normal domain $\Omega$ in a Riemannian manifold $M$, there
exists a complete orthonormal basis
$\{\phi_{1},\phi_{2},\phi_{3},\cdots\}$ of $L^{2}(\Omega)$
consisting of Dirichlet eigenfunctions of the Laplacian $\Delta$,
with $\phi_{j}$ having eigenvalue $\lambda_{j}$ satisfying
\begin{eqnarray*}
0<\lambda_{1}<\lambda_{2}\leq\lambda_{3}\leq\cdots\uparrow\infty.
\end{eqnarray*}
In particular, each eigenvalue has finite multiplicity, and
\begin{eqnarray*}
\phi_{j}\in{C^{\infty}(\Omega)\cap\bar{C}^{1}(\Omega)},
\end{eqnarray*}
where $\bar{C}^{1}(\Omega)$ is the set of functions $v$ satisfying
that $v$ is $C^{1}$ on $\Omega$, and can be extended to a continuous
function on $\overline{\Omega}$, and moreover, the gradient
${\rm{grad}}{v}$ can be extended to a continuous vector field on
$\overline{\Omega}$.

Finally, the heat kernel $H(x,y,t)$ on $\Omega$ satisfies
\begin{eqnarray*}
H(x,y,t)=\sum\limits_{j=1}^{\infty}e^{-\lambda_{j}t}\phi_{j}(x)\phi_{j}(y),
\end{eqnarray*}
with convergence absolute, and uniform, for each $t>0$. In
particular,
\begin{eqnarray*}
\int_{\Omega}H(x,x,t)dV(x)=\sum\limits_{j=1}^{\infty}e^{-\lambda_{j}t}.
\end{eqnarray*}
\end{theorem}

By using Theorem \ref{theoremheat2chapter1} and the comparison
result for the heat kernel, Theorem \ref{7theoremchapter2}, we can
supply another ways to prove the most part of theorems~3.3 and~4.4
in \cite{fmi} -- see Theorem \ref{theorem5} for the precise
statement.

\section{Estimates for the heat kernel}
\renewcommand{\thesection}{\arabic{section}}
\renewcommand{\theequation}{\thesection.\arabic{equation}}
\setcounter{equation}{0} \setcounter{maintheorem}{0}

As before, for a complete $n$-dimensional Riemannian manifold $M$,
denote by $B(p,r_{0})$ the open geodesic ball with center $p$ and
radius $r_0$ of $M$. Let $V_{n}(p^{-} ,r_{0})$ be the geodesic ball
with center $p^{-}$ and radius $r_{0}$ of an $n$-dimensional
spherically symmetric manifold $M^-=[0, l)\times_f\mathbb{S}^{n-1}$
with respect to $p^{-}$, and let $V_{n}(p^{+} ,r_{0})$ be the
geodesic ball with center $p^{+}$ and radius $r_{0}$ of an
$n$-dimensional spherically symmetric manifold $M^+=[0,
l)\times_f\mathbb{S}^{n-1}$ with respect to $p^{+}$, where the model
spaces $M^{+}$ and $M^{-}$ can be determined by the upper and lower
bounds of the radial sectional and Ricci curvatures w.r.t. the given
point $p\in{M}$. This fact has been shown in the previous sections.
Denote by $H(p,y,t)$ the heat kernel on $M$, and by
$H_{+}(p^{+},q,t)$ and $H_{-}(p^{-},q,t)$ the heat kernels on
$M^{+}$ and $M^{-}$, respectively. In this section, we would like to
give an upper and lower bound for the heat kernel. However, before
that, we need to use the following facts in \cite{cy}.

First, we need the following concept, which is used to describe
model spaces considered in \cite{cy}.
\begin{defn} \label{defheatchapter2}
An $n$-dimensional manifold $\mathscr{M}^{n}$ is an open model, if
the following conditions hold:

 (I) For some $x\in\mathscr{M}^{n}$ and $0<R\leq\infty$,
 $\mathscr{M}^{n}=B(x,R)$ (the open ball of radius $R$ about $x$) and
 $\exp_{x}|B_{0}(R)$, with $B_{0}(R)\subseteq{T_{x}}\mathscr{M}^{n}$, is a diffeomorphism.

 (II) For all $r<R$, the mean curvature of the distance sphere
 $S(x,r)$ is constant on $S(x,r)$.
$\\$ Moreover, a model $\mathscr{M}^{n}$ is an open Ricci model if
its metric, when written in polar coordinates, is of the form
\begin{eqnarray*}
dr^{2}+f^{2}(r)h,
\end{eqnarray*}
where $h$ is the standard metric on $\mathbb{S}^{n-1}$. A compact
Riemannian manifold $\mathscr{M}^{n}$ is a closed model (resp.,
closed Ricci model) if, for some $x$,
$\mathscr{M}=\overline{B(x,R)}$ and $B(x,R)$ is an open model
(resp., Ricci model).
\end{defn}

Clearly, by Definition \ref{def1}, we know that a spherically
symmetric manifold must be an open or closed Ricci model with
respect to its base point.

We also need the following lemma which shows us the positivity of
the heat kernel.
\begin{lemma} (\cite{cy}) \label{lemmaheat2chapter2}
Let $\Omega$ be a domain in a Riemannian manifold. Then for either
Dirichlet or Neumann boundary conditions, the heat kernel $H(x,y,t)$
on $\Omega$ satisfies $H(x,y,t)>0$ for $t>0$.
\end{lemma}

By proposition 2.2 and lemma 2.3 of \cite{cy}, we have the following
lemma.
\begin{lemma} (\cite{cy}) \label{lemmaheatchapter2}
(I) Let $\mathscr{M}^{n}$ be an $n$-dimensional open model (with
Dirichlet or Neumann boundary conditions) or a closed model. Then
its heat kernel $H(x,y,t)=H(d(x,y),t)$ depends only on variables
$r:=d(x,y)$ and $t$, with $d$ the distance function on
$\mathscr{M}^{n}$.

(II) Conversely, let $\mathscr{M}^{n}=B(x,R)$ or
$\overline{B(x,R)}$, and assume that $\overline{B(x,R)}$ is
complete. Then if the heat kernel $H(x,y,t)$ depends only on on
variables $r:=d(x,y)$ and $t$, it follows that $\mathscr{M}^{n}$ is
a model.

(III) Let $\mathscr{M}^{n}$ be a model, and let $H(r,t)$ be the
fundamental solution of the heat equation (with respect to Dirichlet
or Neumann boundary conditions if $\mathscr{M}^{n}$ is open). Then,
for all $r,t > 0$, we have
\begin{eqnarray*}
\frac{\partial}{\partial{r}}H(r,t)<0.
\end{eqnarray*}
\end{lemma}

By Lemma \ref{lemmaheatchapter2}, we have the following.
\begin{corollary} \label{corollaryheatchapter2}
For the model space $M^{+}$ (resp., $M^{-}$), its heat kernel
$H_{+}(p^{+},y,t)=H_{+}(r_{1},t)$ (resp.,
$H_{-}(p^{-},y,t)=H_{-}(r_{2},t)$) depends only on variables
$r_{1}:=d_{M^{+}}(p^{+},y)$ (resp., $r_{2}:=d_{M^{-}}(p^{-},y)$) and
$t$, where $d_{M^{+}}$ (resp., $d_{M^{-}}$) denotes the distance
function on $M^{+}$ (resp., $M^{-}$). Moreover, for all $t>0$, we
have
\begin{eqnarray*}
\frac{\partial}{\partial{r_{1}}}H_{+}(r_{1},t)<0,   \qquad
\left({\rm{resp.,}}~
\frac{\partial}{\partial{r_{2}}}H_{-}(r_{2},t)<0\right).
\end{eqnarray*}
\end{corollary}

We also need the following strong maximum (resp., minimum) principle
(cf. \cite{i}, p. 180).
\begin{theorem} \label{71theoremheatchapter2}
Given a Riemannian manifold $M$ with the Laplacian $\Delta$, and the
associated heat operator $L=\Delta-\frac{\partial}{\partial{t}}$.
Let $u(x,t)$ be a bounded continuous function on $M\times[0,T]$,
which is $C^{2}$ with respect to the variable $x\in{M}$, and $C^{1}$
with respect to $t\in[0,T]$, and which satisfies
\begin{eqnarray*}
Lu\geq0  \qquad \left(Lu\leq0\right)
\end{eqnarray*}
on $M\times(0,T)$. If there exists $(x_{0},t_{0})$ in $M\times(0,T]$
such that
\begin{eqnarray*}
u(x_{0},t_{0})=\sup\limits_{M\times[0,T]}u(x,t), \qquad
\left({\rm{resp.,}}~
u(x_{0},t_{0})=\inf\limits_{M\times[0,T]}u(x,t)\right),
\end{eqnarray*}
then
\begin{eqnarray*}
u|M\times[0,t_{0}]=u(x_{0},t_{0}).
\end{eqnarray*}
\end{theorem}

Clearly, the heat equation satisfies both the strong maximum
principle and the strong minimum principle, which implies that the
solution of the heat equation can only achieve its maximum or
minimum on the boundary $\overline{M\times(0,T]}-M\times(0,T]$. One
can easily get a proof of Theorem \ref{71theoremheatchapter2} in
\cite{evance} when $M$ is diffeomorphic to a domain in Euclidean
space. By a standard continuation argument, then one is able to get
a proof for an arbitrary manifold $M$.

By applying Theorems \ref{BishopI} and \ref{BishopII}, Corollary
\ref{corollaryheatchapter2} and Theorem \ref{71theoremheatchapter2},
we can prove the following.

\begin{theorem} \label{7theoremchapter2}
If $M$ is a
 complete n-dimensional Riemannian manifold with a radial sectional curvature
upper bound $k(t)=-\frac{f''(t)}{f(t)}$ w.r.t. a point $p\in{M}$,
then, for $r_{0}<\min\{inj(p),l\}$, we have
 \begin{eqnarray} \label{6.1}
 H(p,y,t)\geq{H_{+}(d_{M^{+}}(p^{+},q),t)}
 \end{eqnarray}
 holds for all $(y,t)\in{B(p,r_{0})\times(0,\infty)}$  with $d_{M^{+}}(p^{+},q)=d_{M}(p,y)$ for any $q\in{M^{+}}$,
 where $d_{M^{+}}$ and $d_{M}$
 denote the distance functions on $M^{+}$ and $M$, respectively. The
 equality in (\ref{6.1}) holds at some
 $(y_{0},t_{0})\in{B(p,r_{0})\times(0,\infty)}$ if and only if
 $B(p,r_{0})$ is isometric to $V_{n}(p^{+},r_{0})$.

 On the other hand, if $M$ is a
 complete $n$-dimensional Riemannian manifold with a radial Ricci curvature
 lower bound $(n-1)k(t)=-\frac{(n-1)f''(t)}{f(t)}$ w.r.t. a point $p\in{M}$, then, for all
$(y,t)\in{B(p,r_{0})\times(0,\infty)}$ and  $r_{0}<\min\{l(p),l\}$
with $l(p)$ defined as in (\ref{important}),  we have
\begin{eqnarray} \label{6.2}
 H(p,y,t)\leq{H_{-}(d_{M^{-}}(p^{-},q),t)}
 \end{eqnarray}
with $d_{M^{-}}(p^{-},q)=d_{M}(p,y)$ for any $q\in{M^{-}}$,
 where $d_{M^{-}}$ and $d_{M}$
 denote the distance functions on $M^{-}$ and $M$, respectively. The
 equality in (\ref{6.2}) holds at some
 $(y_{0},t_{0})\in{B(p,r_{0})\times(0,\infty)}$ if and only if
 $B(p,r_{0})$ is isometric to $V_{n}(p^{+},r_{0})$.
$\\$(The boundary condition will either be Dirichlet or Neumann.)
\end{theorem}

\begin{proof} By the assumptions on curvatures in Theorem \ref{7theoremchapter2},
we know that the model space $M^{+}=[0, l)\times_f\mathbb{S}^{n-1}$
or $M^{-}=[0, l)\times_f\mathbb{S}^{n-1}$ is determined by solving
the initial value problem
\begin{eqnarray*}
\left\{
\begin{array}{lll}
f''(t)+k(t)f(t)=0,\\
f(0)=0,\\
f'(0)=1.  &\quad
\end{array}
\right.
\end{eqnarray*}

Now, assume that the \emph{radial sectional curvature of $M$ is
bounded from above by a continuous function $k(t)=-f''(t)/f(t)$
w.r.t. $p\in{M}$}. By applying Theorem \ref{theoremheatchapter1}, we
have
\begin{eqnarray*}
\lefteqn{H(p,y,t)-H_{+}(d_{M^{+}}(p^{+},q),t)=H(p,y,t)-H_{+}(d_{M}(p,y),t)}\\
&=&H(p,y,t)-H_{+}(r_{1}(p,y),t)\\
&=&\int_{0}^{t}\int_{B(p,r_{0})}\frac{d}{ds}\left[H_{+}(r_{1}(p,z),t-s)H(z,y,s)\right]dV(z)ds\\
&=&-\int_{0}^{t}\int_{B(p,r_{0})}\frac{\partial}{\partial{s}}\left[H_{+}(r_{1}(p,z),t-s)\right]H(z,y,s)dV(z)ds\\
&&+\int_{0}^{t}\int_{B(p,r_{0})}H_{+}(r_{1}(p,z),t-s)\frac{\partial{H}}{\partial{s}}(z,y,s)dV(z)ds\\
&=&-\int_{0}^{t}\int_{B(p,r_{0})}\Delta_{M^{+}}H_{+}(r_{1}(p,z),t-s)H(z,y,s)dV(z)ds\\
&&+\int_{0}^{t}\int_{B(p,r_{0})}H_{+}(r_{1}(p,z),t-s)\Delta_{M}H(z,y,s)dV(z)ds,
\end{eqnarray*}
where $\Delta_{M^{+}}$, $\Delta_{M}$ are the Laplace operators on
$M^{+}$ and $M$, respectively. Since $r_{0}<\min\{inj(p),l\}$, by
applying Green's formula, and using either Dirichlet or Neumann
boundary condition, we have
\begin{eqnarray*}
\int_{B(p,r_{0})}H_{+}(r_{1}(p,y),t)\cdot\Delta_{M}H=\int_{B(p,r_{0})}\Delta_{M}H_{+}(r_{1}(p,y),t)\cdot{H}.
\end{eqnarray*}
So, we obtain
\begin{eqnarray} \label{6.3}
&&H(p,y,t)-H_{+}(d_{M^{+}}(p^{+},q),t)=\int_{0}^{t}\int_{B(p,r_{0})}\big{[}\Delta_{M}H_{+}(r_{1},t-s)-\Delta_{M^{+}}H_{+}(r_{1}(p,y),\nonumber\\
\lefteqn{ \qquad \qquad t-s)\big{]}\cdot H(z,y,s)dV(z)dt.}
\end{eqnarray}
On the other hand, in the geodesic spherical coordinates near $p$ or
$p^{+}$, for function of $r_{1}(p,y)=d_{M^{+}}(p^{+},q)=d_{M}(p,y)$,
we have
\begin{eqnarray*}
&&\Delta_{M^{+}}=\frac{\partial^{2}}{\partial{r_{1}^{2}}}+\frac{\left[f^{n-1}(r_{1})\right]'}{f^{n-1}(r_{1})}\frac{\partial}{\partial{r_{1}}},\\
&&\Delta_{M}=\frac{\partial^{2}}{\partial{r_{1}}^{2}}+\frac{\left[det\mathbb{A}(r_{1},\xi)\right]'}{det\mathbb{A}(r_{1},\xi)}\frac{\partial}{\partial{r_{1}}}=\frac{\partial^{2}}{\partial{r_{1}}^{2}}+\frac{\left(\sqrt{|g|}\right)'}{\sqrt{|g|}}\frac{\partial}{\partial{r_{1}}},
\end{eqnarray*}
where $\mathbb{A}(r_{1},\xi)$ is the path of linear transformations
defined in Section \ref{section2}, and $\sqrt{|g|}$ is defined as
(\ref{J}). So, by Theorem \ref{BishopII}, we have
\begin{eqnarray} \label{6.4}
\Delta_{M}H_{+}(r_{1},t-s)-\Delta_{M^{+}}H_{+}(r_{1}(p,z),t-s)=\left[\frac{\left(\sqrt{|g|}\right)'}{\sqrt{|g|}}-\frac{\left[f^{n-1}(r_{1})\right]'}{f^{n-1}(r_{1})}\right]\frac{\partial{H_{+}}}{\partial{r_{1}}}\geq0.\quad
\end{eqnarray}
Substituting (\ref{6.4}) into (\ref{6.3}), together with Lemma
\ref{lemmaheat2chapter2}, we obtain
\begin{eqnarray*}
H(p,y,t)-H_{+}(d_{M^{+}}(p^{+},q),t)\geq0,
\end{eqnarray*}
which implies (\ref{6.1}). When equality in (\ref{6.1}) holds at
some $(y_{0},t_{0})\in{B(p,r_{0})\times(0,\infty)}$,  by Theorem
\ref{71theoremheatchapter2}, we know that
$H(p,y,t)={H_{+}(d_{M^{+}}(p^{+},q),t)}=H(p,y_{0},t_{0})$ on
$B(p,r_{0})\times[0,t_{0}]$. Together with (\ref{6.4}), we know that
\begin{eqnarray*}
\frac{\left(\sqrt{|g|}\right)'}{\sqrt{|g|}}=\frac{\left[f^{n-1}(r_{1})\right]'}{f^{n-1}(r_{1})}
\end{eqnarray*}
holds on $B(p,r_{0})$. Then by Theorem \ref{BishopII}, we have
\begin{eqnarray*}
\mathbb{A}(r_{1},\xi)=f(r_{1})I
\end{eqnarray*}
for all $r_{1}\leq{r_{0}}$, which implies that $B(p,r_{0})$ is
isometric to $V_{n}(p^{+},r_{0})$.

 Now, assume that the \emph{radial Ricci
curvature of $M$ is bounded from below by a continuous function
$(n-1)k(t)=-(n-1)f''(t)/f(t)$ w.r.t. $p\in{M}$}, and
$r_{0}<\min\{l(p),l\}$. Since the geodesic ball $B(p,r_{0})$ maybe
has points on the cut-locus, which leads to the invalidity of the
path of linear transformations $\mathbb{A}$, we need to use a limit
procedure shown in \cite{cy} to avoid this problem. As the previous
case, by applying Theorem \ref{theoremheatchapter1}, we have
\begin{eqnarray} \label{6.5}
\lefteqn{H(p,y,t)-H_{-}(d_{M^{-}}(p^{-},q),t)=H(p,y,t)-H_{-}(d_{M}(p,y),t)}\nonumber\\
&=&H(p,y,t)-H_{-}(r_{2}(p,y),t)\nonumber\\
&=&-\int_{0}^{t}\int_{B(p,r_{0})}\frac{\partial}{\partial{s}}\left[H_{-}(r_{2}(p,z),t-s)\right]H(z,y,s)dV(z)ds\nonumber\\
&&+\int_{0}^{t}\int_{B(p,r_{0})}H_{-}(r_{2}(p,z),t-s)\frac{\partial{H}}{\partial{s}}(z,y,s)dV(z)ds.
\end{eqnarray}
For any $\xi\in{S^{n-1}_{p}}\subseteq{T_{p}M}$, let
$g(\xi):=\min\{d_{\xi},r_{0}\}$ with $d_{\xi}$ defined in Section
\ref{section2}. Clearly, $g(\xi)$ is a continuous function on the
unit sphere $S^{n-1}_{p}$. As in \cite{cg}, one can choose a
sequence of smooth functions $g_{\epsilon}$ on $S^{n-1}_{p}$, with
$g_{\epsilon}(\xi)<g(\xi)$ for any $\xi\in{S^{n-1}_{p}}$, such that
$g_{\epsilon}$ converges uniformly to $g$ as $\epsilon\rightarrow0$
and the set
\begin{eqnarray*}
V_{\epsilon}=\{\exp_{p}(t\xi)|t\leq{g_{\epsilon}(\xi)}\}
\end{eqnarray*}
is compact. Clearly, $V_{\epsilon}$ is within the cut locus of $p$.
So, the expression (\ref{6.5}) becomes
\begin{eqnarray*}
\lefteqn{H(p,y,t)-H_{-}(d_{M^{-}}(p^{-},q),t)=H(p,y,t)-H_{-}(r_{2}(p,y),t)}\nonumber\\
&=&\lim\limits_{\epsilon\rightarrow0}\Big{\{}-\int_{0}^{t}\int_{V_{\epsilon}}\frac{\partial}{\partial{s}}\left[H_{-}(r_{2}(p,z),t-s)\right]H(z,y,s)dV(z)ds\nonumber\\
&&+\int_{0}^{t}\int_{V_{\epsilon}}H_{-}(r_{2}(p,z),t-s)\frac{\partial{H}}{\partial{s}}(z,y,s)dV(z)ds\Big{\}}\nonumber\\
&=&\lim\limits_{\epsilon\rightarrow0}\Big{\{}-\int_{0}^{t}\int_{V_{\epsilon}}\Delta_{M^{-}}\left[H_{-}(r_{2}(p,z),t-s)\right]H(z,y,s)dV(z)ds\nonumber\\
&&+\int_{0}^{t}\int_{V_{\epsilon}}H_{-}(r_{2}(p,z),t-s)\Delta_{M}H(z,y,s)dV(z)ds\Big{\}},
\end{eqnarray*}
where $\Delta_{M^{-}}$, $\Delta_{M}$ are the Laplace operators on
$M^{-}$ and $M$, respectively. Then, similar to the previous case,
by applying Theorem \ref{BishopI} and Corollary
\ref{corollaryheatchapter2}, we can obtain
\begin{eqnarray} \label{6.6}
\lefteqn{H(p,y,t)-H_{1}(d_{M^{-}}(p^{-},y),t)}\nonumber\\
&=&\lim\limits_{\epsilon\rightarrow0}\left\{\int_{0}^{t}\int_{V_{\epsilon}}\left[\Delta_{M}H_{-}(r_{2},t-s)-\Delta_{M^{-}}H_{-}(r_{2}(p,z),t-s)\right]h(z,y,s)dV(z)dt\right\}\nonumber\\
&=&\lim\limits_{\epsilon\rightarrow0}\left\{\int_{0}^{t}\int_{V_{\epsilon}}\left[\frac{\left[J^{n-1}(r_{2},\xi)\right]'}{J^{n-1}(r_{2},\xi)}-\frac{\left[f^{n-1}(r_{2})\right]'}{f^{n-1}(r_{2})}\right]\frac{\partial{H_{-}}}{\partial{r_{2}}}H(z,y,s)dV(z)dt\right\}\leq0,
\end{eqnarray}
with the function $J(r_{2},\xi)$ defined as (\ref{J}), which implies
(\ref{6.2}). When equality in (\ref{6.2}) holds at some
$(y_{0},t_{0})\in{B(p,r_{0})\times(0,\infty)}$,  by Theorem
\ref{71theoremheatchapter2}, we know that
$H(p,y,t)={H_{-}(d_{M^{-}}(p^{-},q),t)}=H(p,y_{0},t_{0})$ on
$B(p,r_{0})\times[0,t_{0}]$. Together with (\ref{6.6}), we know that
\begin{eqnarray*}
\frac{\left[J^{n-1}(r_{2},\xi)\right]'}{J^{n-1}(r_{2},\xi)}=\frac{\left[f^{n-1}(r_{2})\right]'}{f^{n-1}(r_{2})}
\end{eqnarray*}
holds on $B(p,r_{0})$. Then by Theorem \ref{BishopI}, we have
\begin{eqnarray*}
\mathbb{A}(r_{2},\xi)=f(r_{2})I
\end{eqnarray*}
for all $r_{2}\leq{r_{0}}$, which implies that $B(p,r_{0})$ is
isometric to $V_{n}(p^{-},r_{0})$. Our proof is finished.
\end{proof}

\begin{remark} \rm{ In fact, the completeness of the prescribed
manifold $M$ is a little strong to get the comparison results
(\ref{6.1}) and (\ref{6.2}) for the heat kernel. In \cite{cy},
Cheeger and Yau have shown that if the injectivity radius at some
point $p$ of a prescribed manifold $M$ is bounded from below, then,
under the assumptions on curvature therein, a lower bound can be
given for the heat kernel of geodesic balls on $M$. However, here we
prefer to assume that the prescribed manifold $M$ is complete, since
if $M$ is complete, then for $B(p,r_{0})\subseteq{M}$ with $r_{0}$
finite we can always find optimally continuous bounds for the radial
Ricci and sectional curvatures w.r.t. $p$ (see (\ref{extraII}) and
(\ref{extraIII})). This implies that the assumption on the
completeness of $M$ is feasible.}
\end{remark}

Theorem \ref{theoremheat2chapter1} shows us a connection between the
Dirichlet heat kernel and the Dirichlet eigenvalue of the Laplacian.
Here we would like to use this connection to give another ways to
prove the following Cheng-type eigenvalue inequalities for the
Laplace operator, which have been given in \cite{fmi}.

\begin{theorem} \label{theorem5}
If $M$ is a
 complete n-dimensional Riemannian manifold with a radial Ricci curvature
 lower bound $(n-1)k(t)=-\frac{(n-1)f''(t)}{f(t)}$ w.r.t. a point $p\in{M}$,
then, $r_{0}<\min\{l(p),l\}$ with $l(p)$ defined as in
(\ref{important}), we have
\begin{eqnarray}  \label{6.7}
\lambda_{1}(B(p,r_{0}))\leq\lambda_{1}\left(V_{n}(p^{-},r_{0})\right).
\end{eqnarray}

On the other hand, if $M$ is a
 complete $n$-dimensional Riemannian manifold with a radial sectional curvature
upper bound $k(t)=-\frac{f''(t)}{f(t)}$ w.r.t. a point $p\in{M}$,
then, for $r_{0}<\min\{inj(p),l\}$, we have
\begin{eqnarray}  \label{6.8}
\lambda_{1}\left(B(p,r_{0})\right)\geq\lambda_{1}\left(V_{n}(p^{+},r_{0})\right).
\end{eqnarray}
Here $\lambda_{1}(\cdot)$ in (\ref{6.7}) and (\ref{6.8}) denotes the
first eigenvalue of the corresponding geodesic ball.
\end{theorem}

\begin{proof} Here we would like to use a method similar to that of theorem
1 in p. 104-105 of \cite{rs}. As before, denote separately the
Dirichlet heat kernels of $B(p,r_{0})$, $V_{n}(p^{-},r_{0})$ by
$H(p,y,t)$ and $H_{-}(d_{M^{-}}(p^{-},q),t)$, with
$r_{2}(p,y)=d_{M}(p,y)=d_{M^{-}}(p^{-},q)$, where $d_{M}$ and
$d_{M^{-}}$ are distance functions on $M$ and $M^{-}$, respectively.
If $M$ has \emph{a radial Ricci curvature lower bound
$(n-1)k(t)=-(n-1)f''(t)/f(t)$} w.r.t. $p\in{M}$, and
$r_{0}<\min\{l(p),l\}$, then by Theorem \ref{7theoremchapter2}, we
have
\begin{eqnarray} \label{6.9}
H(p,p,t)\geq{H_{-}(0,t)=H_{-}(r_{2}(p,p),t)}
\end{eqnarray}
for all $t>0$. Furthermore, by Theorem \ref{theoremheat2chapter1},
we can obtain
\begin{eqnarray*}
H(p,p,t)=\sum\limits_{i=1}^{\infty}e^{-\lambda_{i}t}\phi_{i}^{2}(p),\\
H_{-}(0,t)=\sum\limits_{i=1}^{\infty}e^{-\tilde{\lambda}_{i}t}\tilde{\phi}_{i}^{2}(0),
\end{eqnarray*}
with $\lambda_{i}=\lambda_{i}\left(B(p,r_{0})\right)$,
$\tilde{\lambda}_{i}=\lambda_{i}\left(V_{n}(p^{-},r_{0})\right)$,
and $\phi_{i}$, $\tilde{\phi}_{i}$ the corresponding eigenfunctions.
Together with (\ref{6.9}), it follows that
\begin{eqnarray*}
e^{-\lambda_{1}t}\left[\phi^{2}_{1}(p)+e^{-(\lambda_{2}-\lambda_{1})t}\phi^{2}_{2}(p)+\cdots\right]\geq
e^{-\tilde{\lambda}_{1}t}\left[\tilde{\phi}^{2}_{1}(0)+e^{-(\tilde{\lambda}_{2}-\tilde{\lambda}_{1})t}\tilde{\phi}^{2}_{2}(0)+\cdots\right],
\end{eqnarray*}
which is equivalent with
\begin{eqnarray} \label{6.10}
\phi^{2}_{1}(p)+e^{-(\lambda_{2}-\lambda_{1})t}\phi^{2}_{2}(p)+\cdots\geq
e^{(\lambda_{1}-\tilde{\lambda}_{1})t}\left[\tilde{\phi}^{2}_{1}(0)+e^{-(\tilde{\lambda}_{2}-\tilde{\lambda}_{1})t}\tilde{\phi}^{2}_{2}(0)+\cdots\right].
\end{eqnarray}
Since $\phi^{2}_{1}(p)>0$, $\tilde{\phi}^{2}_{1}(0)>0$, and
$\lambda_{m}>\lambda_{1}$ (resp.,
$\tilde{\lambda}_{m}>\tilde{\lambda}_{1}$) for any $m\geq2$, letting
$t\rightarrow\infty$ in (\ref{6.10}) results in
\begin{eqnarray*}
\lambda_{1}-\tilde{\lambda}_{1}\leq0,
\end{eqnarray*}
which implies
\begin{eqnarray*}
\lambda_{1}\left(B(p,r_{0})\right)\leq\lambda_{1}\left(V_{n}(p^{-},r_{0})\right).
\end{eqnarray*}

On the other hand, by applying Theorem \ref{7theoremchapter2} and a
similar method as above, we can easily obtain that for
$r_{0}<\min\{inj(p),l\}$, the inequality
\begin{eqnarray*}
\lambda_{1}\left(B(p,r_{0})\right)\geq\lambda_{1}\left(V_{n}(p^{+},r_{0})\right)
\end{eqnarray*}
holds when $M$ has \emph{a radial sectional curvature upper bound
$k(t)=-f''(t)/f(t)$ w.r.t. $p$}. Our proof is finished.
\end{proof}

\begin{remark} \rm{ In the above proof of Theorem
\ref{theorem5}, when
$\lambda_{1}\left(B(p,r_{0})\right)=\lambda_{1}\left(V_{n}(p^{-},r_{0})\right)$,
we cannot get the characterization, $B(p,r_{0})$ is isometric to
$V_{n}(p^{-},r_{0})$, for this equality as theorem 3.3 in
\cite{fmi}. In fact, if
$\lambda_{1}\left(B(p,r_{0})\right)=\lambda_{1}\left(V_{n}(p^{-},r_{0})\right)$
here, we can only obtain that
$\lim_{t\rightarrow\infty}H(p,p,t)=\lim_{t\rightarrow\infty}H_{-}(0,t)$.
We are not sure whether there exists some $t_{0}<\infty$ such that
$H(p,p,t_{0})=H_{-}(0,t_{0})$ or not, which leads to the fact that
we cannot use the characterization for the equality of (\ref{6.2})
in Theorem \ref{7theoremchapter2}. This can be seen as the
limitation of this new way. The same situation happens to the
equality
$\lambda_{1}\left(B(p,r_{0})\right)=\lambda_{1}\left(V_{n}(p^{+},r_{0})\right)$.}
\end{remark}

\section*{Acknowledgments}
\renewcommand{\thesection}{\arabic{section}}
\renewcommand{\theequation}{\thesection.\arabic{equation}}
\setcounter{equation}{0} \setcounter{maintheorem}{0}
 This research
is supported by Funda\c{c}\~{a}o para a Ci\^{e}ncia e Tecnologia
(FCT) through a doctoral fellowship SFRH/BD/60313/2009. The author
would like to express his gratitude to his Ph.D. advisors, Prof.
Isabel Salavessa and Prof. Pedro Freitas, for suggesting problems
and supplying encouragement and guidance during his doctoral study
at Instituto Superior T\'{e}cnico (IST).

 \end{document}